\RequirePackage{fix-cm}
\documentclass[smallextended]{svjour3}       
\smartqed  
\usepackage{graphicx}
\usepackage{array}
\usepackage{amsmath,amsfonts}

\usepackage{amssymb, mathrsfs}  
\usepackage{latexsym}
\usepackage{graphicx}
\usepackage{mathrsfs}
\usepackage{contour,soul,color}%
\usepackage{booktabs,cases}
\usepackage{tabularx,multirow} 
\usepackage[titletoc]{appendix}
\usepackage{algorithm}
\usepackage{algorithmic}

\usepackage{setspace}
\usepackage[justification=centering,textfont=small,labelfont={bf,small}]{caption}

\usepackage{hyperref}
\hypersetup{
	colorlinks=true,
	linkcolor=blue,
	filecolor=magenta,
	urlcolor=cyan,
	citecolor=blue
}

\newtheorem{myremark}{Remark}

\newcommand{\xqedhere}[2]{%
	\rlap{\hbox to#1{\hfil\llap{\ensuremath{#2}}}}}

\def\Q{{\bar{Q}}}
\def\T{{\bar{T}}}
\def\q{{\bar{q}}}

\def\x{{\bar{x}}}
\def\y{{\bar{y}}}

\def\conv{{\rm conv }\,}

\def\st{\hbox{s.t.}}

\def\[{\begin{equation}}
\def\]{\end{equation}}

\def\d{\displaystyle}
\def\argmax{{\rm argmax}\,}

\def\|{\parallel}

\begin{document}

\title{An exact separation algorithm for unsplittable flow capacitated network design arc-set polyhedron
}

\titlerunning{An exact separation algorithm for unsplittable flow arc-set polyhedron}      

\author{Liang Chen$^{\star}$ \and Wei-Kun Chen$^{\dag}$ \and\\
	 Mu-Ming Yang$^{\star}$ \and Yu-Hong Dai$^{\star}$ \thanks{This work was supported by the Chinese Natural Science Foundation
	 	(Nos. 11631013, 11331012) and the National 973 Program of China (No. 2015CB856002).}}
\authorrunning{Liang Chen et al.} 	
 	
\institute{
	$^{\star}$LSEC, ICMSEC,
	Academy of Mathematics and Systems Science, Chinese Academy of Sciences, Beijing, China \and School of Mathematical Sciences, University of Chinese Academy of Sciences, Beijing, China.
	{\sl Emails}: chenliang@lsec.cc.ac.cn,
	ymm@lsec.cc.ac.cn, dyh@lsec.cc.ac.cn\\
	$^{\dag}$School of Mathematics and Statistics/Beijing Key Laboratory on MCAACI, Beijing Institute of Technology, Beijing, China {\sl Emails}: chenweikun@bit.edu.cn (Corresponding author)
}


\date{Received: date / Accepted: date}

\maketitle

\begin{abstract}
In this paper, we concentrate on generating cutting planes for the unsplittable capacitated network design problem.
We use the unsplittable flow arc-set polyhedron of the considered problem as a substructure and generate cutting planes by solving the separation problem over it.
To relieve the computational burden, we show that, in some special cases, a closed form of the separation problem can be derived.
For the general case, a brute-force algorithm, called  exact separation algorithm, is employed in solving the separation problem of the considered polyhedron such that the constructed inequality guarantees to be facet-defining.
%
%
Furthermore, a new technique is presented to accelerate the exact separation algorithm, which significantly decreases the number of iterations in the algorithm.
Finally, a comprehensive computational study on the unsplittable capacitated network design problem is presented to demonstrate the effectiveness of the proposed algorithm.

\end{abstract}

\section{Introduction}
The unsplittable capacitated network design problem plays an important role in many applications such as telecommunication network design, production distribution, and express package delivery; see \cite{Barnhart2000,Brockmuller2004,Gavish1990} and the references therein.
Given a network, a demand set with its origin-destination pairs of nodes for commodities, and a facility set with different types of facilities (with varying capacities and installation costs), the unsplittable capacitated network design problem is to install integer multiples of facilities on each arc of the network and route the flow of each commodity on a single path such that the total flow cannot exceed the total capacity on each arc, and the sum of facility installation costs and flow routing costs is minimized while meeting the demands of the commodities.

Let $G = (V, E)$ be a directed graph with the node set $V$ and the arc set $E$.
Denote $Q$ and $T$ be the sets of commodities and facilities, respectively.
The demand of commodity $q \in Q$ from the source node $\zeta^q$ to the destination node $\eta^q$ is known as $a_q$.
Suppose that $w^{ij}_q$ is the routing cost for commodity $q$ on arc 
$(i,j)\in E$. 
If one module of facility $t\in T$ is installed on arc $ (i,j) $, let $b_t$ 
and $p^{ij}_t$ represent the additional capacity and the 
installation cost on arc $ (i,j) $, respectively. The existing capacity 
on arc $(i,j)$ is known as $c^{ij}$. We introduce the binary variable $ x^{ij}_q $ to denote whether or not  commodity $ q $ goes through arc $ (i,j) $. The variable $ y^{ij}_t $ denotes the number of facility $ t $ installed on arc $ (i,j) $. With these notations and variables, the mathematical formulation of the unsplittable capacitated network design problem is
\begin{align}
 \d\min_{x,y}\ &\d\sum_{(i,j)\in E}\sum_{q\in Q} w_q^{ij}x_q^{ij}+ \sum_{(i,j)\in E}\sum_{t\in T} p_t^{ij}y_t^{ij},\label{obj}&\\
\st\ & \d\sum_{(i,j)\in E}x^{ij}_q - \!\! \sum_{(j,i)\in E}x^{ji}_q = 
\left\{
\begin{array}{rl}
1,     & \text{if}~ i = \zeta^q;  \\
-1,     & \text{if}~ i = \eta^q;    \\
0,    &  \text{otherwise},     
\end{array} \label{balance}\right.
\!\!\!\!\! &&\forall\,q\in Q,~i\in V,\\
&\d\sum_{q\in Q}a_qx_q^{ij}\leq \sum_{t\in T}b_ty^{ij}_t + c^{ij}, && \forall\,
(i,j)\in E, \label{cap}\\ 
& x^{ij}_q\in \{0,1\},~y_t^{ij} \in \mathbb{Z}_+, \ &&\forall\,q\in Q,\,t\in 
T,\,(i,j)\in E.\label{domain}
\end{align}
In the above formulation, we minimize the sum of facility installation costs and flow routing costs in the objective function \eqref{obj}. Constraint \eqref{balance} is the flow balance constraint.
Constraint \eqref{cap} is the \emph{capacity constraint} that requires that the total flow cannot exceed the total capacity on each arc.

Problem \eqref{obj}-\eqref{domain} is $\mathcal{NP}$-hard even for $|Q|=1$ and $|T|=1$ \cite{CHOPRA1998165}.
Hence there is little hope to develop a theoretically efficient algorithm for solving it. 
Nevertheless, several polyhedral studies of some special cases of the problem have been done in the literature \cite{Achterberg2010,Atamturk2002,Barnhart2000,Benhamiche2016,Brockmuller1996,Brockmuller2004,Gavish1990,Raack2011,vanHoesel2002}, which suggests us that it is possible to develop a computationally efficient algorithm if the polyhedral structure is well understood. Inspired by this, in this paper, we consider the convex hull of the set related to the {capacity constraint} on each arc, i.e., the so-called \emph{unsplittable flow arc-set polyhedron} $ P =  \conv(X) $ where 
$$X = \left \{(x,y)\in\{0,1\}^{|Q|}\times\mathbb{Z}_+^{|T|}: \sum_{q\in 
Q}a_qx_q \leq \sum_{t\in T}b_ty_t+ c\right\}.$$
Here the arc subscripts on variables $x^{ij}_t$ and $ y^{ij}_t$, and parameter $c^{ij}_t$ are dropped.

There exist several works studying the unsplittable flow arc-set 
polyhedron. In particular, Brockm{\"u}ller et al. 
\cite{Brockmuller1996,Brockmuller2004} developed the \emph{c-strong inequality} for 
the unsplittable flow arc-set polyhedron when there are only two facilities and 
the capacity of the second facility is an integer multiple of {that of the first 
one}. 
For problems with a single facility, i.e., $|T|=1$,  Atamt{\"u}rk and Rajan \cite{Atamturk2002} proposed the \emph{$k$-split c-strong inequality} and the \emph{lifted cover inequality}. 
Van Hoesel et al. \cite{vanHoesel2002} studied the \emph{lower convex envelope inequality}. 
Their computational experiments on the c-strong inequality, the $k$-split c-strong inequality, the lifted cover inequality, and the lower convex envelope inequality demonstrate the effectiveness of integrating these inequalities in a branch-and-cut framework.
Benhamiche et al. \cite{Benhamiche2016} generalized the c-strong inequality to solve a variant of the unsplittable capacitated network design problem.

Unfortunately, most of these studies are restricted to the unsplittable flow arc-set polyhedron with a single facility or two facilities with divisible capacities. 
The valid inequalities developed under these assumptions cannot be applied in the context of an arbitrary number of facilities and arbitrary capacities.

In this study, we do not make assumptions on either the number of facilities or the structure of the capacities. 
Instead, our approach is to develop an exact separation algorithm to solve the separation problem of the unsplittable flow arc-set polyhedron $ P $ with an arbitrary number of facilities and arbitrary capacities. 
More precisely, given a point $ (\x,\y)  \in \mathbb{R}^{|Q|}\times \mathbb{R}^{|T|} $, we want to generate a hyperplane to separate point $ (\x,\y) $ from $ P $ or prove that point $ (\x,\y) \in P$.
To do this, we first analyze the coefficients in the nontrivial facet-defining inequality of {polyhedron} $P$, which is employed in formulating the separation problem as an optimization problem. 
We prove that the solution of the optimization problem corresponds to a facet-defining inequality of {polyhedron} $P$. 
To relieve the computational burden, we show that, in some special cases, a closed form of the optimization problem can be derived.
For the general case, the exact separation algorithm, which includes the four steps: preprocessing, row generation, numerical errors, and sequential lifting, is employed in solving the optimization problem.
Furthermore, a new technique is presented to accelerate the exact separation algorithm, which significantly decreases the number of iterations in the row generation subroutine.
Finally, a comprehensive computational study is presented to test the effectiveness of the proposed algorithm.

It is worth noting that the considered exact separation for {the} unsplittable flow arc-set polyhedron can be seen as an extension of the exact separation for {the} 0-1 knapsack polytope; see \cite{Avella2013,Avella2010,Boccia2013,Boyd1993,Boyd1994,Boyd1995,Kaparis2010,Vasilyev2016} and the references therein. 
The difference is that {the} exact separation for {the} 0-1 knapsack polytope cannot handle non-binary integer variables, whereas the approach in this paper takes the non-binary integer variables into consideration such that it can be customized to solve the unsplittable capacitated network design problem.

The organization of this paper is as follows. In Sect. \ref{sect:separationproblem}, we analyze the {properties} of the nontrivial facet-defining inequalities of polyhedron $P$ and formulate the separation problem as an optimization problem.
In Sect. \ref{generalcond}, We consider some special cases for which a closed form of the optimization problem can be derived. 
In Sect. \ref{sect:ExactSeparation}, we give a framework of the exact separation algorithm including preprocessing in Sect. \ref{subsect:Preprocessing}, row generation in Sect. \ref{subsect:RowGeneration}, numerical errors in Sect. \ref{subsect:NumericalErrors}, and sequential lifting in Sect. \ref{subsect:SequentialLifting}.
In Sect. \ref{sect:numericalresults}, we present the numerical results. 
Finally, in Sect. \ref{sect:conclusionandfuturework}, we give some conclusions and future works.

Throughout this paper, let $\boldsymbol{e^i}\in \mathbb{R}^{|Q|}$ and $ 
\boldsymbol{f^j}\in \mathbb{R}^{|T|}$ be the $i$-th $ |Q| $-dimensional unit 
vector and $j$-th $ |T| $-dimensional unit vector, respectively. 
Denote $\boldsymbol e=(1,1,\ldots,1)^{\top}\in \mathbb{R}^{|Q|}$ and $\boldsymbol f = (1,1,\ldots,1)^{\top}\in \mathbb{R}^{|T|}$.  
We use $ X_{\rm{LP}} $ to denote the linear relaxation of {set} $ X $ obtained by relaxing the integer variables to continuous variables.
We assume that $ T = \{1, \ldots, |T|\} \neq \varnothing $, $ 0 < b_1 \leq \cdots \leq b_{|T|} $, and $ a_q > 0$ for all $ q\in Q $.
Without loss of generality, we assume $a^{\top} \boldsymbol e - c >0$ since otherwise the {capacity constraint} in the unsplittable flow arc-set $X$ is redundant.

\section{Separation problem for the unsplittable flow arc-set polyhedron}\label{sect:separationproblem}
In this section, we first study the polyhedral properties of the {unsplittable} flow arc-set polyhedron $P$. 
Then we formulate the separation problem over polyhedron $ P $ as an optimization problem and prove that there exists an optimal solution which corresponds to a facet-defining inequality of polyhedron $P$.

\subsection{Characteristics of the {unsplittable} flow arc-set polyhedron}

We first note that polyhedron $ P $ is full dimensional.
\begin{proposition}\label{dimension}
	The dimension of polyhedron $P$ is $|Q|+|T|$.
\end{proposition}
Next, the following characterizations of some vertices and extreme rays of polyhedron $ P $ are straightforward.
\begin{proposition}\label{extremeray}
	The extreme rays of polyhedron $P$ are $(\boldsymbol{0}, \boldsymbol{f^1}),\ldots,(\boldsymbol{0},\boldsymbol{f^{|T|}})$.
\end{proposition}
\begin{proposition}\label{vertex}
	The point $(x,\rho_t(x)\boldsymbol{f^t})$ is a vertex  of polyhedron $P$ for each 
	$x\in \{0,1\}^{|Q|}$ and $ t \in T $, where 
	\begin{equation}
		\label{rhotxdef}
		\rho_t(x)= \max\left \{\left \lceil (a^{\top}x-c)/b_t\right\rceil,0\right \}.
	\end{equation}
\end{proposition}

{The initial constraints $x_q\geq 0$, $x_q\leq 1$, $y_t\geq 0$, and $a^{\top}x\leq b^{\top}y +c$ are called trivial inequalities of polyhedron $P$.
We now present a necessary condition to guarantee the nontrivial inequality 
 \[\alpha^{\top}x  \leq \beta^{\top}y + \gamma\label{facet}\] 
to be facet-defining for polyhedron $P$.}

\begin{proposition}
	\label{facettheorem}
	Let \eqref{facet} be a nontrivial facet-defining inequality of polyhedron $P$. Then 
	\begin{enumerate}
		\item [\rm{(i)}]$\alpha_q\geq 0$ and $\alpha_q\leq \lceil a_q/b_t\rceil\beta_t$ for each $ q\in Q$ and $  t\in T${\rm{;}}
		\item [\rm{(ii)}]$0< \beta_t\leq \lceil b_t/b_k\rceil\beta_k$ for {each} $t,k\in T$ with $ t\neq k ${\rm{;}}
		\item [\rm{(iii)}]$\gamma\geq \begin{cases}
		0, & \text{if}~c\geq 0{\rm{;}}\\
		-\lceil -c/b_t\rceil \beta_t,~\forall~t\in T, & \text{otherwise}.
		\end{cases}$
	\end{enumerate}\end{proposition}
\begin{proof}
	For notation convenience, denote $F = P \cap \{(x,y) :  \alpha^{\top}x =\beta^{\top}y+ \gamma\}$.
	
	(i) For each $q\in Q$, since $F$ is a nontrivial facet of polyhedron $P$ and 
	inequality \eqref{facet} differs from $x_q\geq 0$, there exists {a point} $(x^{(1)}, y^{(1)})\in F$ with $x^{(1)}_q=1$ and
	\[\alpha^{\top}x^{(1)} = 
	\beta^{\top}y^{(1)} + \gamma.\label{eq:0}\]
	Since the coefficient $ a_q > 0 $, we have $(x^{(1)}-\boldsymbol  
	{e^{q}},y^{(1)})\in P$, and hence point {$(x^{(1)}-\boldsymbol{e^{q}},y^{(1)})$  satisfies  the valid inequality \eqref{facet}, i.e.,} 
	\[\alpha^{\top}x^{(1)} - \alpha_q \leq  
	\beta^{\top}y^{(1)} + \gamma.\label{eq:1}\]  
	Subtracting \eqref{eq:0} from \eqref{eq:1}, we obtain $\alpha_q\geq 0$.
	On the other hand, as inequality \eqref{facet} differs from $x_q\leq 1$, there exists a point $(x^{(2)}, y^{(2)})\in F$ such that $x^{(2)}_q=0$.
	This, combined with the fact that $(x^{(2)}+ \boldsymbol{e^q},y^{(2)}+ \lceil {a_q}/{b_t}\rceil \boldsymbol{f^t})\in P$ for each $ t \in T $, indicates that $\alpha_q \leq \d\lceil a_q/b_t\rceil\beta_{t}$. 
	
	(ii) For each $t\in T$, since inequality \eqref{facet} differs from $y_t\geq 0$, there exists {a point} $(x^{(3)},y^{(3)})\in F$ such that $y^{(3)}_t\geq 1$.
	Then for all $k\in T\backslash\{ t\}$, we have $(x^{(3)}, y^{(3)}-\boldsymbol{f^t}+ \lceil {b_t}/{b_{k}}\rceil \boldsymbol{f^{k}})\in P$,
	which further implies that $\beta_t\leq \lceil b_t/b_k\rceil\beta_k$.	 
	For each $ t \in T $, as $(\boldsymbol 0,\boldsymbol{f^t})$ is an extreme ray of polyhedron $P$ (see Proposition \ref{extremeray}), we have $\beta_t\geq 0$. 
	If $\beta_\tau = 0$ for some $ \tau \in T $, then $0\leq \beta_t \leq \lceil b_t/b_\tau\rceil\beta_\tau=0$ for all $t\in T\backslash \{\tau\} $. 
	Hence, inequality \eqref{facet} reduces to $ \alpha^{\top}x  \leq \gamma $.
	This, together with the fact that $ (\boldsymbol e, \lceil (a^{\top}\boldsymbol e-c)/b_t\rceil \boldsymbol{f^t})\in P$, implies $\alpha^{\top}\boldsymbol e\leq \gamma$. 
	However, this means that inequality \eqref{facet} is dominated by the bound constraints and, thus it cannot define a facet of {polyhedron} $ P $. Therefore, $ \beta_t > 0 $ for all $ t \in T $.
	
	(iii) If $c\geq 0$, since $(\boldsymbol 0,\boldsymbol 0)\in P$, then $\gamma \geq 0$; 
	otherwise, as $(\boldsymbol 0,\lceil -c/b_t\rceil \boldsymbol{f^t} )\in P$, it follows that $\gamma\geq -\lceil -c/b_t\rceil\beta_t$ for all $t \in T$.\qed
\end{proof}

\subsection{Separation problem}
Given a point $(\x,\y)\in \mathbb{R}^{|Q|} \times \mathbb{R}^{|T|}$, the 
separation problem of polyhedron $ P $ is to construct a hyperplane (induced by 
inequality \eqref{facet}) separating $(\x,\y)$ from $P$ strictly, i.e., 
$$\alpha^{\top}x\leq \beta^{\top}y+\gamma ,~\forall~(x,y) \in P,$$
and 
$$\alpha^{\top}\x > \beta^{\top}\y+ \gamma, $$
or prove that no such hyperplane exists, i.e., point $(\x,\y) \in P $. 
The separation problem is trivial to be solved if point $(\x,\y)\notin X_{\rm{LP}}$ since one of the inequalities $ 0 \leq x_q \leq 1 $, $ q \in Q $, $  y_t \geq 0$, $ t \in T $, and $  \sum_{q\in 
	Q}a_qx_q \leq \sum_{t\in T}b_ty_t+ c $ must be violated by this point. Hence, we 
assume that point $(\x,\y) \in X_{\rm{LP}} $ throughout this paper.
Solving the separation problem is equivalent to solving
\begin{equation}
\begin{aligned}
\label{seppro1}
v = \d\max_{\alpha, \beta,\gamma}~& \x^{\top}\alpha -\y^{\top}\beta -\gamma,\\
\st~&  x^{\top}\alpha - y^{\top}\beta -\gamma \leq 0, ~\forall~(x,y) \in P,\\
~& (\alpha,\beta,\gamma)\in S,
\end{aligned}
\end{equation}
where $S$ is a closed convex set which guarantees that problem \eqref{seppro1} is bounded. 
If $v\leq 0$, we prove $(\x,\y) \in P$; otherwise, we find the hyperplane $ \alpha^{\top}x = \beta^{\top}y+\gamma $ separating $(\x,\y)$ from $P$ strictly. 
Let $(x^1,y^1),\ldots,(x^u, y^u)$, $ u \in \mathbb{Z}_+ $, be the vertices of polyhedron $P$. 
From the well-known Minkowski-Weyl theorem \cite{Minkowski,Weyl} and the description of the extreme rays of polyhedron $P$ in Proposition \ref{extremeray}, problem \eqref{seppro1} can be reduced to
\begin{equation}
\label{seppro2}
\begin{aligned}
v = \d\max_{\alpha, \beta,\gamma}~& \x^{\top}\alpha -\y^{\top}\beta -\gamma,\\
\st~&  (x^k)^{\top}\alpha - (y^k)^{\top}\beta -\gamma \leq 0,~k = 1,\ldots,u,\\
~& \beta_t\geq 0,~\forall~t\in T,~ (\alpha,\beta,\gamma)\in S.
\end{aligned}
\end{equation}
Notice that from Proposition \ref{facettheorem}, we have $ \beta_t > 0 $ for all $ t \in T $ in the nontrivial facet-defining inequality \eqref{facet} of {polyhedron} $P$. Hence, here we consider a special choice of $S$:  
\[ S= \{(\alpha,\beta,\gamma): \beta_1 = 1\},\label{normalset}\]
where $ 1 \in T $.
%
Based on this selection, problem \eqref{seppro2} further reduces to 
\begin{equation}
\begin{aligned}
\d\max_{\alpha,\beta,\gamma}~& \x^{\top}\alpha -\y^{\top}\beta - \gamma,\\
\st~& (x^k)^{\top}\alpha - (y^k)^{\top}\beta - \gamma\leq 0,~k = 1,\ldots,u,\\
&\beta_1 = 1,~\beta_t\geq 0,~\forall~	t\in T\backslash \{1\}.
\end{aligned}\label{seppromain}
\end{equation}
Below we shall show that problem \eqref{seppromain} is feasible and bounded.
\begin{lemma}
	Given a point $ (\x, \y) \in X_{\rm{LP}}$, problem \eqref{seppromain} is feasible and bounded. \label{lemma1}
\end{lemma}
\begin{proof}
	We only need to show that problem \eqref{seppromain} is bounded since the point $(\alpha,\beta,\gamma)=(\boldsymbol0, \boldsymbol{f^1}, 0)$ is a feasible solution. 
	We shall use the contradiction argument to prove the boundedness of problem \eqref{seppromain}.
	Suppose that problem \eqref{seppromain} is unbounded.
	Then there exists a vector 
	$(\Delta\alpha, \Delta\beta, \Delta\gamma) 
	\in 
	\mathbb{R}^{|Q|} \times \mathbb{R}^{|T|} \times \mathbb{R}$ such that $ \Delta\beta_1=0  $, $\Delta \beta_t \geq 0$ for 
	all $ t \in T\backslash \{1\} $,
	\[(x^k)^{\top}\Delta\alpha - (y^k)^{\top}\Delta\beta- \Delta\gamma \leq 0,~\text{for all}~k = 1,\ldots,u ,\label{p2}\]
	and
	\[	\x^{\top}\Delta\alpha - \y^{\top}\Delta\beta - \Delta \gamma > 0. \label{p1}\]
	Combining \eqref{p2} and \eqref{p1}, we have
	\[(\x-x^k)^{\top}\Delta\alpha > (\y-y^k)^{\top}\Delta\beta,~\text{for all}~k=1,\ldots,u.\label{ineq}\]
	Define a point $ \hat{x} \in \mathbb{R}^{|Q|} $ such that
	$$\hat{x}_q = \begin{cases}
	1,& \text{if}~\Delta \alpha_q \geq 0;\\
	0,& \text{if}~\Delta \alpha_q < 0.\\
	\end{cases}$$
	From Proposition \ref{vertex}, point $(\hat{x},\rho_1(\hat{x})\boldsymbol{f^1})$ is a vertex of {polyhedron} $P$ where $\rho_1(\hat{x})$ is defined in \eqref{rhotxdef}.
	By substituting this point into \eqref{ineq} and using $0 \leq  \x_q \leq 1$ for all $ q \in Q$, $ \y_t \geq 0$ for all $ t \in T $ (as $ (\x, \y) \in X_{\rm{LP}} $), $\Delta \beta_1 = 0$, and $\Delta \beta_t \geq 0$ for all $t \in T \backslash\{1\}$, we have that 
	\begin{equation*}
	\begin{aligned}
	\d 0 \geq \sum_{q\in Q,\,\Delta\alpha_q\geq 0}(\x_q-1)\Delta\alpha_q + \sum_{q\in Q,\,\Delta\alpha_q< 0}(\x_q-0)\Delta\alpha_q >& \\ 
	 \sum_{t\in T\backslash\{1\}}(\y_t-0)\Delta\beta_t + (\y_1 - \rho_1(\hat{x})) \Delta\beta_1  \geq 0,
	\end{aligned}
	\end{equation*}
	%
	which leads to a contradiction. Thus problem \eqref{seppromain} is bounded.
\qed\end{proof}

Lemma \ref{lemma1} guarantees that problem \eqref{seppromain} contains at least one optimal solution.
Next, we shall show that using the simplex method to solve problem \eqref{seppromain}, we will obtain a facet-defining inequality of {polyhedron} $ P $.

\begin{theorem}
	 Given a point $ (\bar{x}, \bar{y}) \in X_{\rm{LP}}$, the
	 basic optimal 
	 solution $ (\alpha, 
	 \beta, \gamma) $ of the linear programming problem \eqref{seppromain} 
	 defines a facet-defining inequality \eqref{facet} of 
	 polyhedron $P$.
	\label{solisfacet}
\end{theorem}
\begin{proof}
	Let $ (\alpha, \beta, \gamma) $ be a basic optimal solution of problem 
	\eqref{seppromain} and \eqref{facet} be the corresponding inequality. 
	For notation purpose, denote $h = |Q|+ |T|$. 
	From the linear programming theory, except the equality $ \beta_1 = 1 $, 
	there exist 
	another $ h $ active constraints at point $ 
		(\alpha, \beta, \gamma) $ such that the vectors of the coefficients of the constraints are linearly independent.
	Without loss of generality, we assume that these constraints are $ \beta_t = 0 $ for $t = \ell_1, \ldots, \ell_\tau $ ($\ell_i \in T\backslash\{1\}$, $i=1, \ldots, \tau $), and $ (x^k)^{\top}\alpha - (y^k)^{\top}\beta - \gamma= 0 $ for $ k = 1, \ldots, h - \tau $. 
	Then
	$$\begin{bmatrix}
	(x^{1})^{\top}& (y^{1})^{\top} & 1  \\
	\vdots &  \vdots &\vdots  \\
	(x^{h-\tau})^{\top} &(y^{h-\tau})^{\top} &1\\
	{\boldsymbol0}^\top &  (\boldsymbol{f^{\ell_1}})^{\top} & 0  \\
	\vdots &  \vdots &\vdots  \\
	{\boldsymbol0}^\top&  (\boldsymbol{f^{\ell_\tau}})^{\top} & 0  \\
	\end{bmatrix}\begin{bmatrix}
	\alpha\\
	-\beta\\
	 -\gamma
	\end{bmatrix}= \begin{bmatrix}
	0\\
	\\
	\vdots\\
	\\
	\\
	0\\
	\end{bmatrix},$$
	where the rank of the coefficient matrix is $h$. 
	By adding the first row to the last $\tau $ rows, we have 
	$$\begin{bmatrix}
	(x^{1})^{\top}& (y^{1})^{\top} & 1  \\
	\vdots &  \vdots &\vdots  \\
	(x^{h-\tau})^{\top} &(y^{h-\tau})^{\top} &1\\
	(x^{1})^{\top}&  (y^{1}+\boldsymbol{f^{\ell_1}})^{\top} & 1  \\
	\vdots &  \vdots &\vdots  \\
	(x^{1})^{\top}&  (y^{1}+\boldsymbol{f^{\ell_\tau}})^{\top} & 1  \\
	\end{bmatrix}\begin{bmatrix}
	\alpha\\
	-\beta\\
	-\gamma
	\end{bmatrix}= \begin{bmatrix}
	0\\
	\\
	\vdots\\
	\\
	\\
	0\\
	\end{bmatrix}.$$
	It follows immediately that the rank of the new coefficient matrix is also $ h $.
	Hence the points $(x^{1}, y^{1}), (x^{2}, y^{2}), \ldots, (x^{h-\tau}, y^{h-\tau})$, $(x^{1}, y^{1}+\boldsymbol{f^{\ell_1}}), \ldots, (x^{1}, y^{1}+\boldsymbol{f^{\ell_\tau}})$ are affinely independent. 
	Furthermore, these points satisfy \eqref{facet} at equality. This, together with the fact that the dimension of {polyhedron} $P$ is $h$ in Proposition \ref{dimension}, implies that \eqref{facet} is a facet-defining inequality of polyhedron $P$.
\qed\end{proof}

\section{A closed form of the separation problem: sufficient conditions}
\label{generalcond}

Given a point $ (\x, \y) \in X_{\rm{LP}} $, in this section, we consider some special cases for which a closed form of problem \eqref{seppromain} can be derived.
The analysis result of these special cases will be used as a preprocessing technique to reduce the computational time of solving problem \eqref{seppromain} for the general case. 
We first consider the case with a single facility, i.e., $|T|=1$, and then 
generalize it to the multifacility case, 
i.e., $|T|\geq 2$. 
All proofs of the propositions in this section are given in the appendix.

To simplify the notation, in what follows, we refer problem \eqref{seppromain} as the separation problem of polyhedron $ P $. We say inequality \eqref{facet} solves the separation problem \eqref{seppromain} {for polyhedron} $ P $ if point $ (\alpha, \beta, \gamma) $ is one of its optimal solutions. 
For notation purpose, denote
\[{r} =\max\left \{ \left\lceil -  c/b_1\right\rceil, 0\right \} \label{rdef}\]
and 
\[d \in  \argmax_q\left\{\x_q: q\in Q\right \}. \label{ddef}\]

\subsection{Single facility}

In this subsection, we consider the single facility case, i.e., $ T = \{1\} $. 
If $ Q = \varnothing $, problem \eqref{seppromain} is a single variable problem which can be trivially solved, and hence we assume $ Q \neq \varnothing $.
Throughout this subsection, we restrict to consider the case {where} (i) the capacity 
of one module of this facility is larger than or equal to the demand of each commodity; and (ii) 
if any commodity $q$ goes through this arc, we need to install one more  module of this facility on the arc. Mathematically, this can be written as the following two assumptions:
\begin{enumerate}
	\item [(i)]
	$ a_q \leq b_1  $ for all $ q \in Q $;
	\item [(ii)]	
	$ b_1r +  c < a_q $ for all $ q \in Q $.
\end{enumerate}
We next give a closed form of the separation problem \eqref{seppromain} under some conditions in Propositions \ref{propcase1} and \ref{propcase2}.
\begin{proposition}\label{propcase1}
	 Let $(\x,\y) \in X_{\rm{LP}}$. Suppose that $T =\{1\}$, {\rm(i)}, {\rm(ii)}, and 
	\[ \sum_{q \in Q}a_q\leq b_1(r +1) + c	\label{first}\]
	hold.
	The inequality $x_d\leq y_1-r$ solves the separation problem \eqref{seppromain}. Moreover, it defines a facet of polyhedron $ P $.
\end{proposition}
\begin{example}
	Let $X_1 = 
	\left \{(x, y)\in\{0,1\}^4\times\mathbb{Z}_+: 
	11x_1+15x_2+24x_3+50x_4\leq 100y\right \}$ and  $(\x, \y) = 
	(0.3,0.5,0.9,0.1,0.38)$. 
	Since $ r  = 0$, $ c= 0 \leq a_q $ for $ q =1,2,3,4 $, and $ \sum_{q=1}^{4} a_q = 11+15+24+50 = 100 \leq 100= b_1$, by Proposition \ref{propcase1}, inequality $x_3\leq y$ is a solution of the separation problem \eqref{seppromain} which is violated by point $ (\x, \y) $.
\end{example}

The condition \eqref{first} in 
Proposition \ref{propcase1} requires a large $ b $. 
To see this, suppose $  c = 0$. Then $ r = 0 $ and the condition reduces to $ \sum_{q \in Q}  a_q \leq b_1 $, 
which means one module of the facility is enough to cover all the 
demands through this arc. 
Due to this, we derive a closed form under 
a condition with a smaller $b$ in the following.

\begin{proposition}\label{propcase2}
	Let $(\x,\y) \in X_{\rm{LP}}$. Suppose that $T =\{1\}$, {\rm(i)}, {\rm(ii)}, and
	\begin{equation}
	\label{third}
	\sum_{q \in Q}a_q  - a_{\q} \leq  b_1(r+1)+ c< \sum_{q \in Q}a_q, \ \forall 
	~ \q \in Q
	\end{equation}
	hold.
	Then we have the followings.
	\begin{enumerate}
		\item 
		[ {\rm(a)}]  If $ |Q| \leq 2$, the inequality  $\sum_{q\in Q}x_q\leq y_1-r$  solves the separation problem \eqref{seppromain}.
		\item 
		[{\rm(b)}] If $ |Q| \geq 3 $, one of the following three inequalities solves the separation problem  \eqref{seppromain}, respectively:
		$$\left\{\begin{aligned}
		&x_{d}\leq  y_1-r,\ && \!\!\!\!\!\text{if}\  \frac{\sum_{q\in Q}\x_q}{|Q|-1} \leq 
		\x_d;\\
		&\frac{1}{|Q|-1}\sum_{q \in Q} x_q\leq y_1 -r,\ && \!\!\!\!\!\text{if}\ \x_d < 
		\frac{\sum_{q\in Q}\x_q}{|Q|-1}\leq 1;\\
		&\sum_{q\in Q}x_q\leq  y_1-r+|Q|-2,\ &&\!\!\!\!\!\text{if}\ \frac{\sum_{q\in 
		Q}\x_q}{|Q|-1}> 1.
		\end{aligned} \right.$$ 
	\end{enumerate}
	Moreover, in both two cases, the inequalities define facets of polyhedron $ P $, respectively.
\end{proposition}

\begin{example}
	Let $X_2 = \left \{(x, y)\in\{0,1\}^4\times\mathbb{Z}_+: 
	11x_1+15x_2+24x_3+50x_4\leq 90y\right\}$. 
	It is easy to verify the conditions of Proposition \ref{propcase2} are satisfied. 
	Suppose $(\x, \y) = (0.4,0.5,0.4,0.4,0.47)$. By simple calculation, it 
	follows that $ d=2 $, $ \bar{x}_d = 0.5$, and $0.5 < \sum_{i=1}^{4}\x_i/(4-1) = {\frac{17}{30}} < 1$.
	From Proposition \ref{propcase2}, this gives us the inequality  $$\frac{1}{3}(x_1+x_2+x_3+x_4)\leq y,$$
	which cuts off point $ (\x, \y) $.
\end{example}

\subsection{Multifacility}
In this subsection, we consider the multifacility case, i.e., $ |T| \geq 2 $. 
Throughout this subsection, we restrict to consider the case with (i), (ii), and 
\begin{enumerate}
	\item [(iii)]	
	$ \sum_{q\in Q} a_q \leq b_t +  c$ for all $ t\in T\backslash\{1\}$.
\end{enumerate}
Assumption (iii) means that except the facility $1 \in T$, the capacity of one module of other facilities is large enough to carry out all commodities.

Given {a point} $ (\x, \y) \in X_{\rm{LP}} $, under assumption (iii), we observe that if $\sum_{t\in T\backslash\{1\}}\y_t\geq 1$, it follows that point $ (\x,\y) \in P $.
Indeed, since point $ (\bar{x}, \bar{y}) \in X_{\text{LP}} $, we have $ \boldsymbol{0} \leq \bar{x} \leq \boldsymbol{1} $. Then, using Proposition \ref{facettheorem}(i) and the fact that $ a_q > 0 $, in order to prove $ (\bar{x}, \bar{y}) \in P $, it is enough to show that point $ (\boldsymbol{e}, \bar{y}) \in P $.
	Similarly, by Proposition \ref{facettheorem}(ii) and $ b_t > 0 $ for all $ t \in T$, it suffices to show that $ (\boldsymbol{e}, \bar{y}) \in P $ with $ \sum_{t \in T \backslash\{1\}} \bar{y}_t = 1$ and $ \bar{y}_1 = 0 $. 
	The latter is true since $ (\boldsymbol{e}, \bar{y}) = \sum_{t \in T \backslash\{1\}} \bar{y}_t  (\boldsymbol{e}, \boldsymbol{f^t}) $, and by assumption (iii), $ (\boldsymbol{e}, \boldsymbol{f^t}) \in P $ for each $ t \in T \backslash \{1\} $. Therefore, in the remaining of this subsection, we only consider the case $\sum_{t\in T\backslash\{1\}}\y_t< 1$.
	
Let 
\[\tilde{Q}:=\bigg \{q \in Q:\x_q > \sum_{t\in T\backslash\{1\}}\y_t\bigg \}\label{Qtilde}.\]
Based on Propositions \ref{propcase1} and \ref{propcase2}, we can derive similar results under the additional assumption (iii) in the case $ |T| \geq 2 $. This is summarized in Propositions \ref{propcase3} and \ref{propcase4}. 

\begin{proposition}\label{propcase3}
	Let $(\x,\y) \in X_{\rm{LP}}$. Suppose that $|T| \geq 2$, 
	{\rm(i)}, {\rm(ii)}, {\rm(iii)}, \eqref{first} and $\sum_{t\in T\backslash\{1\}}\y_t < 1$ hold. Then we have the followings.
	\begin{itemize}
	\item[\rm{(a)}] If $ \tilde{Q} = \varnothing $, the inequality $0 \leq y_1 + r \sum_{t\in T\backslash\{1\}} y_t -r$ solves the separation problem \eqref{seppromain}.
	\item[\rm{(b)}] If $ \tilde{Q} \neq \varnothing $, the inequality $x_d\leq 
	y_1 + (r+1)\sum_{t\in T \backslash\{1\}}y_t-r$ solves the separation problem \eqref{seppromain}. 
	\end{itemize}
Moreover, in both two cases, the inequalities define facets of polyhedron $ P $, respectively.
\end{proposition}

\begin{proposition}\label{propcase4}
	Let $(\x,\y) \in X_{\rm{LP}}$. Suppose that $|T| \geq 2$, 
	{\rm(i)}, {\rm(ii)}, {\rm(iii)}, 
	\eqref{third}, and $\sum_{t\in T\backslash\{1\}}\y_t < 1$ hold. Then we have the followings.
	\begin{itemize}
		\item[\rm{(a)}] If $ \tilde{Q} =\varnothing $, the inequality $0 \leq y_1 + r \sum_{t\in T\backslash\{1\}} y_t -r$ solves the separation problem \eqref{seppromain}.
		\item[\rm{(b)}] If $ \tilde{Q}\neq \varnothing$ and $\tilde{Q} \neq Q$, 
		the inequality $x_d\leq y_1 + (r+1)\sum_{t\in T \backslash\{1\}}y_t-r$ solves the separation problem \eqref{seppromain}.
		\item[\rm{(c)}] If $\tilde{Q} = Q$ with  $|Q|\leq 2$, the inequality $\sum_{q\in Q}x_q\leq 
		y_1 + (r+|Q|)\sum_{t\in T \backslash\{1\}}y_t-r$ solves the separation problem \eqref{seppromain}.
		\item[\rm{(d)}] If $\tilde{Q} = Q$ with  $|Q|\geq 3$, one of the following three inequalities solve the separation problem \eqref{seppromain}, respectively:
		$$\small\!\!\!\!\!\!\!\!\!\!\!\left\{\begin{aligned}
		&\!x_{d}\leq y_1 + 
		(r+1)\sum_{t\in T\backslash\{1\}} y_t-r,\ && \!\!\!\!\!\!\!\text{if}\  \frac{\sum_{q\in 
		Q}\x_q-\sum_{t\in T\backslash\{1\}}\y_t}{|Q|-1} \leq \x_d;\\
		&\!\frac{1}{|Q|-1}\sum_{q \in Q} x_q\leq y_1+ 
		\bigg(r+\frac{|Q|}{|Q|-1}\bigg)\sum_{t\in T\backslash\{1\}} y_t-r,\ &&\!\!\!\!\!\!\!\text{if}\ 
		\x_d < \frac{\sum_{q\in Q}\x_q-\sum_{t\in T\backslash\{1\}}\y_t}{|Q|-1}\leq 1;\\
		&\!\sum_{q\in Q}x_q\leq y_1+(r+2)\sum_{t\in T\backslash\{1\}} 
		y_t-r+|Q|-2,\ &&\!\!\!\!\!\!\!\text{if}\ \frac{\sum_{q\in 
		Q}\x_q-\sum_{t\in T\backslash\{1\}}\y_t}{|Q|-1}> 1.
		\end{aligned} \right.$$
	\end{itemize}
	Moreover, in all four cases, the inequalities define facets of polyhedron $ P $, respectively.
\end{proposition}

\begin{myremark}
Together with the trivial inequalities, the inequalities listed in Propositions {\ref{propcase1}-\ref{propcase4}} describe polyhedron $ P $, respectively.
Otherwise, suppose that there exists a facet-defining inequality \eqref{facet} with $ \beta_1 = 1 $ differing from any {of} the inequalities {in the} list {of} Proposition \ref{propcase1} (Proposition \ref{propcase2}, \ref{propcase3}, or  \ref{propcase4}).
Then there exists a point $ (\x, \y) \in P $ fulfilling inequality \eqref{facet} at equality and the inequalities listed in the proposition are all inactive at this point.
Considering the separation problem \eqref{seppromain} with point $ (\x, \y) $, we know that the inequalities listed in the proposition cannot solve this  problem, which leads to a contradiction. 
\end{myremark}

\section{An exact separation algorithm}\label{sect:ExactSeparation}

Unlike Sect. \ref{generalcond}, in this section, we focus on solving problem \eqref{seppromain} without any assumptions on the data. 
%
%
To begin with, we note that since the number of constraints in problem  \eqref{seppromain} may be exponential, 
from a computational perspective, it is impractical to solve problem \eqref{seppromain} when all the constraints are expressed explicitly.
For this reason, we follow \cite{Vasilyev2016} to solve problem \eqref{seppromain} by decomposing into the following four steps. 

\begin{enumerate}
	\item {\bf Preprocessing}. In order to save computational time, we 
	implement some preprocessing methods before solving problem 
	\eqref{seppromain}. Firstly, instead of solving {problem} \eqref{seppromain} 
	directly, we solve a lower dimensional problem to 
	find a violated inequality for the 
	lower dimensional polyhedron.
	Besides, we avoid {solving problem \eqref{seppromain}} if a most violated inequality is known or point $(\x,\y) \in P$ (e.g., by using Propositions \ref{propcase1}-\ref{propcase4}).

	\item {\bf Row generation.} We find a violated inequality or report that no violated one exists by solving problem \eqref{seppromain} over the lower dimensional polyhedron with a row generation subroutine. 
	\item {\bf Numerical errors}. To avoid numerical instabilities, the constructed inequality is scaled to obtain integral coefficients
	and the right hand side is recomputed to guarantee its validity.
	\item {\bf Sequential lifting}. The variables that are fixed in the preprocessing step are sequentially lifted to obtain a strong valid inequality for polyhedron $ P $.		
\end{enumerate}

\subsection{Preprocessing}\label{subsect:Preprocessing}
Given a point $ (\x, \y) \in X_{\rm{LP}} $, {let us fix the variables which take values on their bounds} and consider the lower dimensional polyhedron $ P(\x, \y) = \conv (X(\x,\y))$ where
$$
X(\bar{x},\bar{y})  =  
\left\{(x,y)\in\{0,1\}^{|\Q|}\times\mathbb{Z}_+^{|\T|}:\sum_{q\in 
\Q}a_qx_q\leq \sum_{t\in \T}b_ty_t +\bar{c} \right\},$$ 
$\Q =\{q\in Q: 0<\x_q<1 \}$, $\T = \{t\in T: \y_t>0 \}$, and $\bar{c} = c-\sum_{q\in Q,\,\x_q=1}a_q$. 
Similar to problem \eqref{seppromain}, we may solve the separation problem over {polyhedron} $ P(\x, \y) $ by considering 
\[\begin{aligned}
\d\max_{\alpha, \beta,\gamma}\ & \d\sum_{q\in \Q}\x_q\alpha_q - \sum_{t\in \T}\y_t\beta_t - \gamma,\\
\st\ & \d\sum_{q\in \Q} x^k_q\alpha_q - \d\sum_{t\in \T}y^k_t\beta_t -\gamma \leq 0,\ k =1 ,\ldots, \bar{u},\\
&\beta_1 = 1,~\beta_t\geq 0,\ \forall\ t\in \T\backslash\{1\},
\end{aligned}\label{lowdimLP}\]
where $(x^1,y^1),\ldots, (x^{\bar{u}},y^{\bar{u}})$, $ \bar{u} \in \mathbb{Z}_+ $, are the vertices of polyhedron $P(\x, \y)$.
Notice that here we, without loss of generality, assume that $ \y_1 >0 $.
Both the numbers of variables and constraints in problem \eqref{lowdimLP} are less {than or equal to those in problem \eqref{seppromain}}. Therefore, it can be expected that problem \eqref{lowdimLP} is easier to be solved than problem \eqref{seppromain}, especially when the number of fixed variables is large.

We next present a preprocessing procedure based on the following observations.
\begin{enumerate}
	\item [(i)] If $ \Q = \varnothing $ and $ \T = \varnothing $, then point $ (\x, \y) \in P(\x,\y) $ and, hence there does not exist a violated inequality.
	\item [(ii)] If $ \Q = \varnothing $ and $ \T = \{1\} $, there is only a single facet-defining inequality $ y_1 \geq \lceil 
	-\bar{c}/b_1 \rceil $ of polyhedron $ P(\x,\y) $. If it is violated by {point} $ (\x,\y) $, we directly move to the sequential lifting step in Sect. \ref{subsect:SequentialLifting}. Otherwise, we have $ (\x, \y) \in P (\x, \y) $.
	\item [(iii)] If 
	\begin{equation}
		\label{processingcondition3}
	\sum_{q\in \Q}a_q\lceil\x_q\rceil \leq  \sum_{t\in 
	\T}b_t\lfloor\y_t \rfloor+\bar{c},
	\end{equation} 
	then we can conclude that $ (\x, \y) \in P (\x, \y) $. To see this, let $$  K = \left \{(\x, \y): 
	x_q \in \{\lfloor 
	\x_q\rfloor, \lceil \x_q\rceil\},\ \forall q\in \Q,~y_t\in\{\lfloor 
	\y_t\rfloor, \lceil \y_t\rceil\},\ \forall t\in\T\right \}.$$
	It follows from  \eqref{processingcondition3} that $ K \subseteq X(\x, \y) $. Therefore, we have $ (\x, \y) \in \conv(K) \subseteq P(\x, \y)$.  

	\item [(iv)] For {polyhedron} $ P(\x, \y) $, if one of the conditions in 
	Propositions \ref{propcase1}-\ref{propcase4} is satisfied, we move to the sequential lifting step in Sect. \ref{subsect:SequentialLifting} or conclude that point $ (\bar{x}, \bar{y}) \in P(\x, \y) $ depending on whether or not one of the inequalities listed in the corresponding proposition is violated by point $ (\x, \y) $. 
\end{enumerate}

\subsection{Row generation}\label{subsect:RowGeneration}
We now describe the row generation subroutine for solving problem \eqref{lowdimLP}.
Instead of solving the whole problem \eqref{lowdimLP} with the potentially exponential many constraints, the row generation subroutine solves a problem with a subset of constraints, i.e.,
\[\begin{aligned}
v(U) = \d\max_{\alpha, \beta,\gamma}\ & \d\sum_{q\in \Q}\x_q\alpha_q  - \sum_{t\in \T}\y_t\beta_t -\gamma,\\
\st\ & \d \sum_{q\in \Q}x_q\alpha_q -\sum_{t\in \T}y_t\beta_t - \gamma \leq 0,\ \forall~(x,y)\in U ,\\
& \beta_t\geq 0,\ \forall\ t\in \T\backslash\{1\}, \  \beta_1 = 1,
\end{aligned}\label{partialLP}\]
in each iteration, where $ U \subseteq P(\x, \y)$.  
We call \eqref{partialLP} the {\it partial separation problem}. Apparently, the 
partial separation problem \eqref{partialLP} is a relaxation of problem \eqref{lowdimLP}.
Therefore, if $v(U) \leq 0$, we conclude that point $(\x,\y)\in P(\x, \y)$; otherwise, the solution $ (\bar{\alpha}, 
\bar{\beta}, \bar{\gamma}) $ of problem \eqref{partialLP} corresponds to the inequality
\begin{equation}\label{lowerdimvalidineq}
	\sum_{q\in \Q}\bar{\alpha}_qx_q \leq \sum_{t\in \T}\bar{\beta}_ty_t+ \bar{\gamma},
\end{equation}
which is violated by {point} $ (\x, \y) $.
To further test whether or not inequality \eqref{lowerdimvalidineq} is valid for $ X(\x, \y) $, 
we solve the following unbounded integer knapsack problem:  
\[ z= \d\max_{x,y}\left \{ \d\sum_{q\in\Q}\bar{\alpha}_qx_q -\sum_{t\in\T}\bar{\beta}_ty_t- \bar{\gamma}: (x,y) \in X(\x, \y)\right \}.\label{subproblem}\]
If $z\leq 0$, inequality \eqref{lowerdimvalidineq} is valid for $ X(\x, \y) $; otherwise the optimal solution of \eqref{subproblem} violates inequality \eqref{lowerdimvalidineq}, and hence, we add this solution into set $ U $ and the procedure continues.

In contrast to the separation problem \eqref{seppromain}, {for} which it is shown to be bounded in Lemma \ref{lemma1}, the partial separation problem \eqref{partialLP} can be unbounded.
\begin{example}
	Let $X^0= \left \{( x, y)\in \{0,1\}\times\mathbb{Z}_+: 3x\leq 5y\right \}$ and $( \x, 
	\y) = (0.5,0.3)$. 
	Initializing $ U = \varnothing $, the partial separation problem \eqref{partialLP} reduces to
	$v(\varnothing) = \max_{\alpha,\gamma} \left \{0.5\alpha-0.3-\gamma\right \}$, which is unbounded. 
\end{example} 
In order to avoid the partial separation problem \eqref{partialLP} to be unbounded, we shall provide some bounds on the variables. 
As it has been shown in Theorem \ref{solisfacet}, the basic optimal solution of problem \eqref{lowdimLP} corresponds to a facet-defining inequality of {polyhedron} $ P(\x, \y) $. 
Combining it with Proposition \ref{facettheorem} and the fact that $ \beta_1=1 $, we can add the bound constraints
	\begin{equation}
	\label{bounds}
	\begin{aligned}
	& &&  0 \leq\alpha_q\leq \lceil a_q/ b_1\rceil,  &&\forall ~q\in \Q,\\
	& && 1 \leq\beta_t \leq \d\lceil b_t/b_{1}\rceil, 
	&&\forall ~t\in \T,\\
	& && \min\left \{ -\lceil -\bar{c}/b_1\rceil, 0\right \}\leq\gamma,&&
	\end{aligned}
	\end{equation}
to problem \eqref{lowdimLP} and then derive a stronger partial separation problem:
\[\begin{aligned}
v(U) = \d\max_{\alpha, \beta,\gamma}\ & \d\sum_{q\in \Q}\x_q\alpha_q  - \sum_{t\in \T}\y_t\beta_t -\gamma,\\
\st\ & \d \sum_{q\in \Q}x_q\alpha_q -\sum_{t\in \T}y_t\beta_t - \gamma \leq 0,\ 
&&\forall\ (x,y)\in U ,\\
&  0\leq  \alpha_q \leq \lceil {a_q}/{b_1}\rceil,\ &&\forall\ q\in\Q,\\
&  1  \leq \beta_t\leq \lceil {b_t}/{b_1}\rceil,\ &&\forall\ 
t\in  \T,\\
& \gamma\geq \min\{ -\lceil -\bar{c}/b_1\rceil, 0\}.
\end{aligned}\label{partialLP1}\]
Clearly, problem \eqref{partialLP1} is bounded.

Another issue needed to be addressed is that the optimal solution of the unbounded knapsack problem \eqref{subproblem} may not be unique. The inequality derived by some optimal solution may be stronger than those derived by other optimal solutions. 
\begin{example}
	\label{ex3}
	Let $X_3 = \left \{(x, y)\in\{0,1\}^4\times\mathbb{Z}_+: 
	11x_1+15x_2+24x_3+50x_4\leq 60y\right \}$. Consider the point $(\x,\y) = (0.9, 0.5, 0.7, 0.1, 0.7)$.
	After solving the partial separation problem \eqref{partialLP1} with $U = \{(1,1,1,1,2)\}$, we obtain the solution $ (\alpha, \beta, \gamma) = (1,0,1,0,1, 0) $ corresponding to the inequality $x_1+x_3\leq y$. Considering
	the associated unbounded integer knapsack problem \eqref{subproblem}, we have two optimal solutions 
	$(1,0,1,0,1)$ and $ (1,1,1,0,1)$ which correspond to the inequalities 
	\[\alpha_1+\alpha_3 -1 -\gamma \leq 0\label{firstineq} \]
	and
	\[\alpha_1+ \alpha_2+\alpha_3 -1 -\gamma \leq 0, \label{secondineq}\]
	respectively. Obviously, inequality \eqref{secondineq} is stronger than inequality \eqref{firstineq}.  
	If we add inequality \eqref{secondineq} to 
	problem \eqref{partialLP1}, in the next iteration, we will obtain the optimal solution $  (\alpha, \beta, \gamma) = (1,0,0,1,1,0) $ corresponding to the inequality $x_1+x_4\leq y$, which is violated by {point} $ (\x, \y) $.
	Moreover, by solving unbounded integer knapsack problem \eqref{subproblem}, we know that $x_1+x_4\leq y$ is valid for $ P(\x,\y) $. However, if we add inequality 
	\eqref{firstineq} to problem \eqref{partialLP1}, it can be checked that it needs more iterations to solve the separation problem \eqref{lowdimLP}.
\end{example}

We see from Example \ref{ex3} that, to speed up the row generation procedure, among (possible) multiple optimal solutions of the unbounded integer knapsack problem, it is crucial to select one which corresponds to a stronger inequality for problem \eqref{partialLP1}. Due to this, we next describe an iterative approach to modify an existing optimal solution for problem \eqref{subproblem} such that the new optimal solution corresponds to a (possibly) stronger constraint in problem \eqref{partialLP1}.

Let $(\tilde{x},\tilde{y})$ be an optimal 
solution of the unbounded integer knapsack problem \eqref{subproblem}. The corresponding inequality in problem \eqref{partialLP1} is 
\[\sum_{q\in \Q}\tilde{x}_q\alpha_q  -\sum_{t\in 
	\T}\tilde{y}_t\beta_t - 
\gamma \leq 0. \label{cons2}\]
Now suppose that $\tilde{x}_{q'} = 0$ for some $ q' \in \Q $ and point  $(\tilde{x}+\boldsymbol{e^{q'}},\tilde{y})\in X(\x,\y)$. 
Notice that as point $(\bar{\alpha}, \bar{\beta}, \bar{\gamma})$ is feasible solution {of} problem \eqref{partialLP1}, we have $ \bar{\alpha}_{q'} \geq 0 $. 
Hence $(\tilde{x}+\boldsymbol{e^{q'}},\tilde{y})$ is also an optimal solution of problem \eqref{subproblem} corresponding to 
the inequality \[\sum_{q\in \Q}\tilde{x}_q\alpha_q + \alpha_{q'}-\sum_{t\in 
	\T}\tilde{y}_t\beta_t - \gamma 
\leq 0, \label{cons1}\] 
which is obviously stronger than inequality \eqref{cons2}. 
Furthermore, we can recursively use this argument to strengthen an
inequality based on the new optimal solution. We describe this iterative approach in 
Algorithm \ref{acceleration}.

\begin{algorithm}[!htbp]
	\caption{A procedure to obtain a stronger constraint for problem \eqref{partialLP1}}
	\label{acceleration}
	\begin{algorithmic}[1]
		\REQUIRE The set $ X_{LP}(\x, \y) $ and the point $ (\tilde{x}, \tilde{y}) \in X(\x, \y)$. 
		\ENSURE A new point $(\tilde{x}, \tilde{y})$ corresponding to a (possible) 
		stronger inequality of problem \eqref{partialLP1}.
		\STATE Reorder $a_1\geq a_2 \geq\cdots\geq a_{|\Q|} $;
		\FOR{ $q=1,...,|\Q|$}
		\IF{ $\tilde{x}_q = 0 $ and $a^{\top}\tilde{x} + a_q \leq 
			b^{\top}\tilde{y}+\bar{c}$ }
		{
			\STATE $\tilde{x}_q \leftarrow 1$;
		}
		\ENDIF
		\ENDFOR
	\end{algorithmic}
\end{algorithm}

In Algorithm \ref{acceleration}, we first sort the variables such that $a_1\geq a_2 \geq\cdots\geq a_{|\Q|} $ in Step 1.
We then recursively modify a point to a new point which corresponds to a (possibly) stronger constraint for problem \eqref{partialLP1} in Steps 3-5.

To conclude this subsection, we present our row generation procedure in Algorithm \ref{rowgeneration}.
\begin{algorithm}[!htbp]
	\caption{Row generation procedure}
	\label{rowgeneration}
	\begin{algorithmic}[1]
		\REQUIRE The set $ X(\x, \y) $ and the point $ (\bar{x}, \bar{y}) \in X_{\rm{LP}}(\x, \y) $.
		\ENSURE Find a violated inequality \eqref{lowerdimvalidineq} for polyhedron $P(\x,\y)$ or conclude that point $(\bar{x},\bar{y}) \in P(\x, \y)$.
		\STATE Initialize the paritial separation problem \eqref{partialLP1} with $U =\varnothing$;
		\STATE Solve the paritial separation problem \eqref{partialLP1} with the solution $(\bar{\alpha}, \bar{\beta}, \bar{\gamma})$ and the optimal value $v(U)$;
		\STATE  If $v(U) \leq 0$, stop and conclude that $(\bar{x},\bar{y}) \in P(\x, \y)$;
		\STATE Solve the unbounded integer knapsack problem \eqref{subproblem} with the solution $(\tilde{x}, \tilde{y})$  and the optimal value $\tilde{z}$;
		\STATE  If $\tilde{z}> 0$, {modify the solution $(\tilde{x}, \tilde{y})$ using Algorithm \ref{acceleration}} and set $U \leftarrow U \cup 
		(\tilde{x}, \tilde{y})$. Go to Step 2;
		\STATE Return the violated inequality \eqref{lowerdimvalidineq} for polyhedron $P(\x, \y)$;
	\end{algorithmic}
\end{algorithm}

In Algorithm \ref{rowgeneration}, we first initialize the set $U = \varnothing$ in Step 1. We then solve problem \eqref{partialLP1} in Step 2 using the dual simplex method with the warm start information in the last step, see for example \cite{Koberstein2005}. 
In Step 3, the optimal value of problem \eqref{partialLP1} is nonpositive, and hence point $(\bar{x},\bar{y}) \in P(\x, \y)$. 
In Step 4, we modify the code BOUKNAP, written by Pisinger \cite{Pisinger2000}, to solve the unbounded integer knapsack problem \eqref{subproblem} with fractional value costs.
In Step 5, we know the current inequality is invalid for $P(\x,\y)$. 
Therefore, we first modify the solution (obtained in Step 4) using Algorithm \ref{acceleration} such that it corresponds to a (possibly) stronger inequality. Then we add the corresponding inequality to problem \eqref{partialLP1} and go to Step 2.
Finally, in Step 6, if $\tilde{z} \leq 0$, we obtain the valid inequality 
\eqref{lowerdimvalidineq}, which is violated by point $(\x,\y)$.

\subsection{Numerical errors}\label{subsect:NumericalErrors}
Due to the numerical errors incurred in solving problem \eqref{partialLP1}, we may get an invalid inequality in Algorithm \ref{rowgeneration}.
To avoid this, we scale the inequality to obtain integral coefficients by solving the integer programming problem
$$
\begin{aligned}
\d\min_{\alpha,\beta,\theta} ~&\theta,\\
\st~& \bar{\beta}  \theta = \beta,\\
~& \bar{\alpha}  \theta = \alpha,\\
~& \beta_t\in \mathbb{Z},\, ~\forall~t\in\T,\\	
~& \alpha_q\in \mathbb{Z}, ~\forall~q\in\Q,\\
~& \theta \geq 1.
\end{aligned}
$$
Vasilyev et al. \cite{Vasilyev2016} suggested to use the enumeration of the multiplier $\theta$ from $2$ to $10^4$ and checked the integrality within the tolerance $10^{-5}$. 
To further avoid too much computational efforts, here we use the approach described in \cite{Achterberg2009}.
More formally, let $\frac{\lambda_q}{\mu_q}$ be the rational representation of coefficient $\bar{\alpha}_q$.
Note that $\frac{\lambda_q}{\mu_q}$ can be obtained by using the Euclidean algorithm within a small tolerance. 
To avoid too large numbers, we give the following requirements:
$$\bigg |\frac{\lambda_q}{\mu_q} - \bar{\alpha}_q \bigg|\leq \epsilon := 10^{-9},~|\lambda_q|\leq \lambda_{\max} := 10^6,~\text{and}~\mu_q\leq \mu_{\max} := 10^3.$$
The representation of $\bar{\beta}_t$ for each $ t \in T $ is also required to be satisfied with the same restrictions. 
Let $ \mu $ be the least common multiple of all denominators. We also require $\mu \leq \mu_{\max}$.
We scale $ \bar{\alpha} $ and $ \bar{\beta} $ by setting $\alpha_q = \mu\bar{\alpha}_q$ for all $q\in\Q$ and $\beta_t = \mu\bar{\beta}_t$ for all $ t\in\T$  with the requirement that $|\mu\alpha_{q}|\leq \mu_{\max}$ and $|\mu\beta_{t}|\leq \mu_{\max}$. 
If all the requirements are satisfied, we accept this inequality; otherwise, we drop it. 

After scaling $ \alpha_q $ for all $ q \in Q $ and $ \beta_t $ for all $ t \in T $, we recompute the right hand side $\gamma $ by solving the unbounded integer knapsack problem 
\[\gamma = \max_{x,y}\left \{\sum_{q\in \Q}\alpha_qx_q - \sum_{t\in \T}\beta_ty_t  : \d (x,y) \in X(\x, \y) \right \}.\label{initiallifting}\]
This leads to the inequality 
 \begin{equation}
 	\label{partialvalidineq}
 	\sum_{q \in \Q} \alpha_q x_q \leq \sum_{t \in \T} \beta_t y_t + {\gamma},
 \end{equation}
 which is valid for $P(\x, \y)$.

\subsection{Sequential lifting}
\label{subsect:SequentialLifting}
The inequality \eqref{partialvalidineq} is valid for $ P(\x, \y) $. However, in general, it {may be} invalid for $ P $. 
To resolve this problem, the variables, {which are} fixed in the preprocessing step, are sequentially lifted according to a given lifting order $ \Pi $, i.e., a permutation of $ (Q \backslash \Q) \cup (T \backslash \T)$. 
We now illustrate the procedure to lift the first variable with index $ k $ based on inequality \eqref{partialvalidineq}.
Denote
$$W(C) = \max_{(x,y) \in 
	\{0,1\}^{|\Q|} \times \mathbb{Z}_+^{|\T|} } \left \{\sum_{q\in \Q}\alpha_qx_q  - \sum_{t\in \T}\beta_ty_t: 
\d  \sum_{q \in \Q}a_q x_q  \leq \sum_{t \in \T} b_t y_t + C\right\}.$$ 
 We have the following three cases.
\begin{itemize}
	\item [1)]
If $ k \in Q \backslash \Q $ and $\x_k = 0$, then the lifted inequality is
 \begin{equation}
\label{liftineq1}
\sum_{q\in \Q}\alpha_qx_q + \alpha_kx_k \leq \sum_{t\in \T}\beta_ty_t+ {\gamma},
 \end{equation}
where $\alpha_k ={\gamma} - W(\bar{c}-a_k).$
\item [2)] If $ k \in  Q \backslash \Q $ and $\x_k = 1$, then the lifted inequality is 
\begin{equation}
\label{liftineq2}
\sum_{q\in \Q}\alpha_qx_q + \alpha_k(x_k-1) \leq \sum_{t\in \T}\beta_ty_t +{\gamma},
\end{equation}
where 
$\alpha_k = W(\bar{c}+a_k)-{\gamma}.$
\item [3)]If $ k \in T \backslash \T  $ with $\y_k = 0$, then the lifted inequality is
\begin{equation}
\label{liftineq3}
\sum_{q\in \Q}\alpha_qx_q \leq \sum_{t\in \T}\beta_ty_t+\beta_ky_k+{\gamma},
\end{equation}
where 
$$
\beta_k=\max\left\{(W(\bar{c}+\ell b_k)- {\gamma})/\ell:\ell = 1,\ldots,\bar{\ell}\right \}.
$$
Here we set $\bar{\ell}=\lceil (\sum_{q\in \Q}a_q-\bar{c})/b_k\rceil $ since $W(\bar{c}+\ell b_k) = \sum_{q\in \Q}\alpha_q$ for all $\ell \geq \bar{\ell}$. 
\end{itemize}
Similarly, we may continue to lift the other variables, which are fixed in the  preprocessing step in Sect. \ref{subsect:Preprocessing}, to obtain a valid inequality of polyhedron $P$.
During the whole lifting process, we need to solve several integer knapsack problems. 
This can be done via the dynamic programming algorithm. 
For more details, we refer to \cite{Vasilyev2016}.

\section{Numerical results}\label{sect:numericalresults}

In order to test the effectiveness of the exact separation algorithm for solving the unsplittable capacitated network design problem, we 
implement it in C++ linked with IBM ILOG CPLEX 
optimizer 12.7.1 \cite{Cplex}
library.
Following \cite{Vasilyev2016}, to 
avoid changing the problem structure, the presolving features are turned off in our experiments.
Moreover, the dual simplex method is used to reoptimize the linear programming problem \eqref{partialLP1} after adding cutting planes.
To eliminate the effect of multithreads, the computations are implemented 
in a single thread.
The time limit is set to 7200 seconds.
Except where explicitly stated, the other parameters in CPLEX are set to the default ones. 
The exact separation procedure stops if the optimal value of the LP relaxation problem of the unsplittable capacitated network design problem improves by less than $0.01\%$ between two adjacent calls.

\subsection{Testsets}\label{testsets}

We conduct our computational study on three testsets of the unsplittable capacitated network design problem \eqref{obj}-\eqref{domain}.
The first testset NDP1, studied in Atamt\"{u}rk et al.  \cite{Atamturk2002}, contains $ 20 $ instances with a single facility. We use this testset to compare the performance effect of the cuts constructed in our exact separation procedure with that of existing cuts studied in the literature. 
The second testset NDP2 includes 26 realistic network instances generated by 
the Survivable Network Design Library (SNDlib 1.0) \cite{SNDlib10}. 
9 of them 
are instances with a single facility and 17 of them are instances with 
multifacility. %
To possibly reduce the unstable behavior of integer programming solvers (see, 
e.g., \cite{Fischetti2014,Andrea2013}), we solve each instance using 10 
different random seeds in 7200 seconds. 
We treat every pair of instance and 
seed as an individual model, which results in 200 models for testset 
{NDP1}  and 260 models for testset {NDP2} . 

The third testset NDP3 is randomly generated based on \cite{Luo2019}.
We use this testset to evaluate the performance effects of different capacity module sizes and module costs on the unsplittable capacitated network design problem using standard integer programming solver or our exact separation procedure.
Table \ref{capandcost} lists different capacity module sizes and module costs studied in \cite{Luo2019}. 
For the same facility, we assume that its module costs on all arcs are the same. In total, we study problems with $ 27 $ different capacities and costs structures.
We generate the underlying graphs with $ 50 $ vertices using the procedure described in \cite{Luo2019}; see also \cite{Magnanti1995,Salman2008}.
We generate 10 graphs with 20  commodities with random source nodes and destination nodes.
The demand of each commodity is chosen uniformly in $ \{10,11,\ldots, 190 \} $.
For each graph, $ 27 $ problems are generated based on each item of capacity module sizes and costs in Table \ref{capandcost}. Thus, in total, we have $ 270 $ models for testset NDP3.

\begin{table}[!htbp]
	\centering
	\setlength{\tabcolsep}{3pt} 
	\renewcommand{\arraystretch}{1.2} 
	\centering
	{\small
		\centering
		\caption{Capacity module sizes and module costs of testset NDP3.}
		\begin{tabular}{|c|l|l|}
			\hline
			& Capacity module sizes & Capacity module costs \\
			\hline
			$1\_1\_1$ & (130) & (10000) \\
			$2\_1\_1$ & (130,50) & (10000,5000) \\
			$3\_1\_1$ & (130,50,20) & (10000,5000,2500) \\
			\hline
			$1\_1\_2$  & (130) & (18000) \\
			$2\_1\_2$  & (130,50) & (18000,9000) \\
			$3\_1\_2$ & (130,50,20) & (18000,9000,5000) \\
			\hline
			$1\_1\_3$  & (130) & (25000) \\
			$2\_1\_3$  & (130,50) & (25000,13000) \\
			$3\_1\_3$ & (130,50,20) & (25000,13000,9000) \\
			\hline
			$1\_2\_1$ & (170) & (10000) \\
			$2\_2\_1$ & (170,70) & (10000,5000) \\
			$3\_2\_1$ & (170,70,30) & (10000,5000,2500) \\
			\hline
			$1\_2\_2$ & (170) & (18000) \\
			$2\_2\_2$ & (170,70) & (18000,9000) \\
			$3\_2\_2$ & (170,70,30) & (18000,9000,5000) \\
			\hline
			$1\_2\_3$  & (170) & (25000) \\
			$2\_2\_3$ & (170,70) & (25000,13000) \\
			$3\_2\_3$ & (170,70,30) & (25000,13000,9000) \\
			\hline
			$1\_3\_1$ & (200) & (10000)\\
			$2\_3\_1$ & (200,80) & (10000,5000)\\
			$3\_3\_1$ & (200,80,30) & (10000,5000,2500) \\
			\hline
			$1\_3\_2$ & (200) & (18000)\\
			$2\_3\_2$ & (200,80) & (18000,9000)\\
			$3\_3\_2$ & (200,80,30) & (18000,9000,5000) \\
			\hline
			$1\_3\_3$  & (170) & (25000) \\
			$2\_3\_3$ & (200,80) & (25000,13000)\\
			$3\_3\_3$ & (200,80,30) & (25000,13000,9000) \\
			\hline
		\end{tabular}
		\label{capandcost}
	}
\end{table}

\subsection{Different lifting orders}
In the sequential lifting step, different lifting orders lead to different inequalities \cite{Vasilyev2016}. 
Hence the first experiment is conducted to test performance effect of different lifting orders. 
We consider the following four lifting orders.

\begin{itemize}
	\item [$\bullet$]{LIFT1}: Variables, fixed to one, are lifted first 
	in decreasing order of their coefficients. Then variables, fixed 
	to zero, are lifted also in decreasing order of their coefficients.
	\item [$\bullet$]{LIFT2}: Variables are lifted in decreasing order of their 
	coefficients.
	\item [$\bullet$]{LIFT3}: Variables are lifted in decreasing order of their 
	reduced costs.
	\item [$\bullet$]{LIFT4}: Variables are lifted in increasing order of their 
	reduced costs.
\end{itemize}

\begin{table}[!htbp]
	\centering
	\centering
	\setlength{\tabcolsep}{3pt} 
	\renewcommand{\arraystretch}{1.2} 
	\small
	\caption{Performance comparison of different lifting orders.}
	\begin{tabular}{|c|cccc|}
		\hline
		Testset& {LIFT1} & {LIFT2} & {LIFT3} & {LIFT4}\\
		\hline
		{NDP1}   &71.51&71.51&71.52& 71.54\\ 
		{NDP2}  &54.39& 54.49&54.12&54.96\\ 
		{NDP3}  &87.35& 87.29&87.21&87.40\\ 
		\hline
	\end{tabular}\label{gapclosed}
\end{table}

Table \ref{gapclosed} presents the 
arithmetic means of the \emph{gap closeds} \cite{Wolter2006} of all models in the corresponding testsets. 
The {gap closed} is defined as 
$$100 \cdot \d\frac{z_{\rm{root}}-z_{\rm{LP}}}{z_{\rm{ub}} - z_{\rm{LP}}},$$
where $z_{\rm{root}}$ is the objective value of the LP relaxation at root node after 
adding cuts, $z_{\rm{LP}}$ is the value of the LP relaxation before adding 
cuts, and $z_{\rm{ub}}$ is the value of the best known value of the model. 
Table \ref{gapclosed} shows that using the lifting order LIFT4, the gap closed is slightly better than those of other lifting orders in all the three testsets. 
Therefore, in further computational studies, the lifting order {LIFT4} is used in our exact separation algorithm.

\subsection{Performance effect of the exact separation procedure}
In this subsection, we evaluate the performance effect of adding cuts generated by {the} exact 
separation algorithm into the solver.

\begin{table}[!htbp]
	\centering
	\setlength{\tabcolsep}{3pt} 
	\renewcommand{\arraystretch}{1.2} 
	{\small
		\centering
		\caption{Performance comparison of the exact separation procedure with the default setting for testset {NDP1}.}
		\begin{tabular}{|l|c|ccc|ccccc|}
			\hline
			\multirow{2}{1cm}{Bracket}&\multirow{2}{0.6cm}{total}	&	
			\multicolumn{3}{c|}{CPX}					&	
			\multicolumn{5}{c|}{EXACT}									
			\\\cline{3-10}
			&		&	solved	&	nodes	&	time	&	solved	
			&	nodes	
			&	time	&	faster	&	slower	\\\hline
			all	&	200	&	145	&	1313	&	32	&	151	&	1105	&	31	
			&	34	&	79	\\
			$[1,7200]$	&	125	&	116	&	2968	&	48	&	122	&	2476	
			&	46	&	34	&	65	\\
			$[10,7200]$	&	61	&	52	&	33410	&	304	&	58	&	29352	
			&	274	&	23	&	21	\\
			$[100,7200]$	&	34	&	25	&	296061	&	2266	&	31	&	
			219844	&	1733	&	21	&	5	\\
			$[1000,7200]$	&	24	&	15	&	423246	&	4142	&	21	&	
			356030	&	3469	&	13	&	4	\\\hline
		\end{tabular}\label{unsplit}
	}
\end{table}

\begin{table}[!htbp]
	\centering
	\setlength{\tabcolsep}{3pt} 
	\renewcommand{\arraystretch}{1.2} 
	{\small
		\centering
		\caption{Performance comparison of the exact separation procedure with the default setting for testset {NDP2}.}
		\begin{tabular}{|l|c|ccc|ccccc|}
			\hline
			\multirow{2}{1cm}{Bracket}&\multirow{2}{0.6cm}{total}	&	
			\multicolumn{3}{c|}{CPX}					&	
			\multicolumn{5}{c|}{EXACT}									
			\\\cline{3-10}
			&	&	solved	&	nodes	&	time	&	solved	
			&	nodes	
			&	time	&	faster	&	slower	\\\hline
			all	&	260	&	55	&	26062	&	291	&	63	&	21904	&	269	
			&	33	&	20	\\
			$[1,7200]$	&	63	&	55	&	26062	&	291	&	63	&	21904	
			&	269	&	33	&	20	\\
			$[10,7200]$	&	56	&	48	&	38772	&	427	&	56	&	31884	
			&	391	&	30	&	17	\\
			$[100,7200]$	&	49	&	41	&	48949	&	595	&	49	&	
			38996	&	531	&	27	&	14	\\
			$[1000,7200]$	&	19	&	11	&	113188	&	3410	&	19	&	
			86696	&	2738	&	12	&	6	\\\hline
		\end{tabular}\label{sndlib}
	}
\end{table}

\begin{table}[!htbp]
	\centering
	\setlength{\tabcolsep}{3pt} 
	\renewcommand{\arraystretch}{1.2} 
	{\small
		\centering
		\caption{Performance comparison of the exact separation procedure with the default setting for testset {NDP3}.}
		\begin{tabular}{|l|c|ccc|ccccc|}
			\hline
			\multirow{2}{1cm}{Bracket}&	\multirow{2}{0.6cm}{total}	&	
			\multicolumn{3}{c|}{CPX}				&	
			\multicolumn{5}{c|}{EXACT}			
			\\\cline{3-10}
			& 	&	solved	&	nodes	&	time	&	solved	
			&	nodes	
			&	time	&	faster	&	slower	\\\hline
			all	&	270	&	139	&	165616	&	1180	&	210	&	9172	&	145	&	196	&	7	\\
			$[1,7200]$	&	210	&	139	&	165616	&	1180	&	210	&	9172	&	145	&	196	&	7	\\
			$[10,7200]$	&	203	&	132	&	202814	&	1374	&	203	&	10758	&	158	&	193	&	4	\\
			$[100,7200]$	&	178	&	107	&	331671	&	2211	&	178	&	16789	&	212	&	173	&	1	\\
			$[1000,7200]$	&	125	&	54	&	704511	&	4781	&	125	&	31822	&	359	&	120	&	1	\\\hline
		\end{tabular}\label{Luo2019}
	}
\end{table}

Tables \ref{unsplit}-\ref{Luo2019} compare the computational results obtained by using CPLEX (CPX) and our exact separation procedure (EXACT). We report the number of solved models, the average running time, and the average number of explored nodes\footnote{Shifted 
	geometric mean, 10s for average time and 100 for average nodes 
	\cite{Achterberg2009}.}.
Besides, columns ``faster" and ``slower" report
the number of models that get at least 10\% faster and slower, respectively.
We group 
the three testsets into several brackets. 
The bracket ``all" contains the models which can be solved by at least one of the settings. 
The bracket $[n,7200]$ contains the models which can be solved by the 
slower setting in at least $ n $ seconds. 
The larger $n$ is, the harder of 
the model is. 
For each bracket, we report the number of considered models in column ``total".

As it can be observed in Tables \ref{unsplit}-\ref{Luo2019}, our exact separation algorithm has a positive effect on all these three testsets, especially on the hard models. In particular, using the exact separation algorithm, we can solve $ 6  $, $8$, and $ 71 $ more models than those using the default setting on testsets NDP1, NDP2, and NDP3, respectively.
This clearly shows that our exact separation algorithm can improve the performance of the solver on solving the unsplittable capacitated network design problem.

Specifically, for testset NDP1, using the exact separation algorithm, we have 79 models for which the running times are slower than that of the {default setting} while only 34 models are solved faster. 
This is due to the fact that these models are easier than models in testsets NDP2 and NDP3, and, as a result, the benefit of the exact separation algorithm cannot compensate for its additional overhead on these easy models.

For testset NDP2, we notice that these models are harder than those in testset NDP1 or NDP3. Among these 260 models, CPX and EXACT only solve 55 and 63 of them, respectively. Nevertheless, for the solved models, using EXACT, the average running time decreases from 291s to 269s. 

\begin{table}[!htbp]
	\centering
	\setlength{\tabcolsep}{3pt} 
	\renewcommand{\arraystretch}{1.2} 
	{\small
		\centering
		\caption{Results of different capacity module sizes and module costs on testset NDP3.}
		\begin{tabular}{|c|ccc|ccc|}
			\hline	
			\multirow{2}{1cm}{}&\multicolumn{3}{c|}{CPX} 
			&\multicolumn{3}{c|}{EXACT}\\\cline{2-7}
			&	nodes	&	time	&	solved	&	nodes	&	time	&	
			solved	\\\hline
			1\_1\_1	&	101498	&	1107	&	6	&	52605	&	744	&	7	
			\\
			2\_1\_1	&	133392	&	833	&	7	&	6596	&	96	&	9	\\
			3\_1\_1	&	757202	&	2451	&	5	&	2857	&	55	&	9	
			\\
			\hline
			1\_1\_2	&	115713	&	1274	&	6	&	48061	&	710	&	7	
			\\
			2\_1\_2	&	120158	&	794	&	7	&	6767	&	101	&	8	\\
			3\_1\_2	&	361194	&	1367	&	6	&	2927	&	66	&	9	
			\\
			\hline
			1\_1\_3	&	98817	&	1081	&	6	&	50658	&	709	&	7	
			\\
			2\_1\_3	&	135529	&	892	&	6	&	7409	&	99	&	8	\\
			3\_1\_3	&	280378	&	1178	&	7	&	4800	&	78	&	10	
			\\
			\hline
			1\_2\_1	&	77609	&	1228	&	6	&	44647	&	987	&	6	
			\\
			2\_2\_1	&	221946	&	1313	&	6	&	21803	&	217	&	9	
			\\
			3\_2\_1	&	364549	&	1700	&	5	&	4373	&	43	&	10	
			\\
			\hline
			1\_2\_2	&	82223	&	1270	&	6	&	44890	&	973	&	6	
			\\
			2\_2\_2	&	157895	&	1096	&	6	&	22698	&	242	&	8	
			\\
			3\_2\_2	&	487028	&	2328	&	5	&	4140	&	47	&	10	
			\\
			\hline
			1\_2\_3	&	87375	&	1373	&	5	&	42778	&	889	&	6	
			\\
			2\_2\_3	&	165320	&	1115	&	6	&	21347	&	244	&	9	
			\\
			3\_2\_3	&	321156	&	1579	&	8	&	16831	&	149	&	10	
			\\
			\hline
			1\_3\_1	&	247636	&	4362	&	4	&	134158	&	3010	&	
			7	\\
			2\_3\_1	&	334434	&	3353	&	3	&	77329	&	1019	&	
			7	\\
			3\_3\_1	&	507180	&	3190	&	4	&	14150	&	165	&	10	
			\\
			\hline
			1\_3\_2	&	223731	&	3916	&	3	&	125455	&	2897	&	
			5	\\
			2\_3\_2	&	283211	&	2889	&	3	&	90339	&	1325	&	
			6	\\
			3\_3\_2	&	536980	&	3464	&	3	&	14477	&	199	&	9	
			\\
			\hline
			1\_3\_3	&	231183	&	4150	&	3	&	144692	&	3252	&	
			5	\\
			2\_3\_3	&	285635	&	2936	&	3	&	122664	&	1513	&	
			5	\\
			3\_3\_3	&	338583	&	2674	&	4	&	33673	&	460	&	8	
			\\
			\hline
		\end{tabular}
		\label{Luo2019details}
	}
\end{table}

For testset NDP3, Table \ref{Luo2019} shows a significant improvement of using our exact separation algorithm.
The average running 
time decreases from 1180s to 145s. 
To further see where the improvement comes from, we report the computational results of different capacity module sizes and costs of facilities independently in Table 
\ref{Luo2019details}.
To be more specific, we list the results in $ 9 $ different groups corresponding to Table \ref{capandcost}.
In each group, the number of facilities is different in each item (i.e. $ i\_j\_k $). 
There are $ 10 $ different models in each item with different network structures and commodities. 
We report the number of average nodes, the average running time, and the number of solved models in each item.
In the same group, we can observe that with our exact separation algorithm, the models with more facilities are easier to be solved than those with fewer facilities. 
However, the same behavior cannot be observed in {the computational results of} CPX.  
This shows that compared with problems with a single facility, our exact separation algorithm works better in problems with more facilities.

\begin{figure}[!htbp]
	\centering
	\includegraphics[width=9cm]{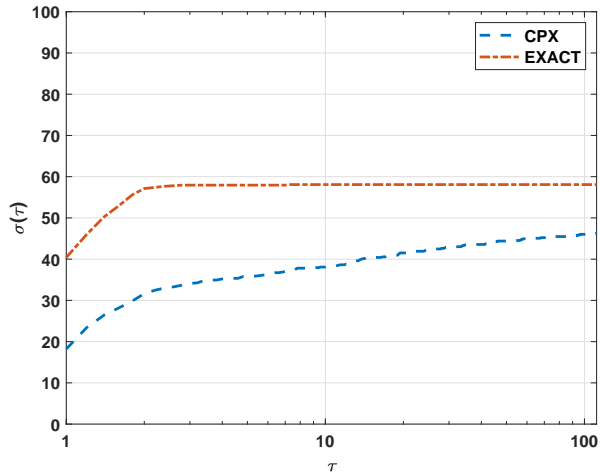}
	\caption{Performance profile for testsets NDP1, NDP2, and NDP3.}
	\label{fig:Performanceprofile}
\end{figure}

To end of this subsection, we plot the performance profiles of CPX and EXACT in Fig. \ref{fig:Performanceprofile} to further compare their performance.
For each model in the testsets NDP1, NDP2, and NDP3, we compute a factor $\tau$ as the ratio of the running time to solve to optimality of the considered setting to the minimum running time of two settings CPX and EXACT.
Each point $(\tau, \sigma)$ of the curves in Fig. \ref{fig:Performanceprofile} represents that in $\sigma$ percentage of the models, this particular setting is at most $\tau$ times slower than the {faster setting}; for more details, see \cite{Dolan2002}.
Consequently, the higher the curve is, the better the setting performs.
Fig. \ref{fig:Performanceprofile} clearly shows that EXACT performs much better than CPX.  
In particular, EXACT is able to solve $58.1\%$ of the models to optimality while CPX is only able to solve $46.3\%$ of the models to optimality.

\subsection{Comparison with c-strong inequality \cite{Brockmuller1996,Brockmuller2004}}
According to \cite{atamturk2017multi}, the c-strong inequality \cite{Brockmuller1996,Brockmuller2004} is quite effective in solving the unsplittable capacitated network design problem.
Other proposed inequalities \cite{Atamturk2002,vanHoesel2002} can provide additional improvement but the marginal effect on top of the c-strong inequality is limited. 
Therefore, we only compare the performance effects of the inequality generated by the exact separation procedure (EXACT) with the c-strong inequality (CSTRONG). 
Since the c-strong inequality can only be applied in the case of one facility or two facilities with divisible capacities, we only conduct experiments on testset NDP1. 
The results are reported in Table \ref{cstrong-exact-all}.

\begin{table}[!htbp]
		\centering
	\setlength{\tabcolsep}{3pt} 
	\renewcommand{\arraystretch}{1.2} 
	{\small
		\centering
		\caption{Performance comparison with the c-strong inequalities on testset NDP1.}
		\begin{tabular}{|l|c|ccc|ccccc|}
			\hline
			\multirow{2}{1cm}{Bracket}	&	\multirow{2}{0.6cm}{total}	&	
			\multicolumn{3}{c|}{CSTRONG}					&	
			\multicolumn{5}{c|}{EXACT}									
			\\\cline{3-10}
				&		&	solved	&	nodes	&	time	&	solved	&	
				nodes	&	time	&	faster	&	slower	\\\hline
			all	&	200	&	147	&	1232	&	31	&	151	&	1033	&	28	&	53	&	67	\\
			$[1,7200]$	&	130	&	125	&	2281	&	41	&	129	&	1883	&	38	&	45	&	62	\\
			$[10,7200]$	&	60	&	55	&	31508	&	275	&	59	&	25637	&	232	&	27	&	22	\\
			$[100,7200]$	&	32	&	27	&	366689	&	2374	&	31	&	222227	&	1585	&	20	&	4	\\
			$[1000,7200]$	&	22	&	17	&	579321	&	4725	&	21	&	377860	&	3246	&	13	&	3	\\\hline
		\end{tabular}
		\label{cstrong-exact-all}
	}
\end{table}
Table \ref{cstrong-exact-all} shows that the performance of EXACT is better than that of CSTRONG, especially on the hard models. In total, EXACT solves $ 4 $ more models than CSTRONG. 
This shows that even compared with the existing polyhedral studies on the unsplittable flow arc-set polyhedron, our exact separation algorithm is more effective in solving {the} unsplittable capacitated network design problem.

\subsection{Performance effect {of} using Propositions \ref{propcase1}-\ref{propcase4}}

{We now report the performance effect of Propositions \ref{propcase1}-\ref{propcase4} in our procedure in Table \ref{easycalls}.} The average results over all the models in each testset are presented.
Compared with EXACT, {NOPROS} refers to the setting of calling the row generation subroutine to solve the separation problem \eqref{lowdimLP} even if one of the conditions in Propositions \ref{propcase1}-\ref{propcase4} is satisfied.
For each setting, we report the running time spent in the
row generation subroutine in column ``rgtime".
For EXACT, we additionally report the numbers of times that fulfilling the conditions in Propositions \ref{propcase1}-\ref{propcase4} in ``p5"-``p8", respectively. 
 In column ``ncalls", we report the number of calling the exact separation algorithm. In column ``rate'', we list the successful rate computed by the number of times fulfilling at least one of the conditions in Propositions \ref{propcase1}-\ref{propcase4} over the number of calling the exact separation algorithm.
\begin{table}[ht]
	\centering
	\setlength{\tabcolsep}{3pt} 
	\renewcommand{\arraystretch}{1.2} 
	\centering
	\small 
	\caption{Performance effect {of} using Propositions \ref{propcase1}-\ref{propcase4}.}
	\begin{tabular}{|c|c|ccccccc|}
		\hline
		\multirow{2}{1cm}{Testset}& \multicolumn{1}{c|}{NOPROS} & 
		\multicolumn{7}{c|}{EXACT} \\\cline{2-9}
		& rgtime & rgtime & p5 & p6 & 
		p7 & p8 & ncalls & rate  \\ 
		\hline
		{NDP1} &1.21&0.80 &255& 543 & 0 & 0 &  1158 &  66.42\% \\
		{NDP2} &102.91 & 91.10 & 183	 & 151	& 10& 411& 1361&  31.72\% \\
		{NDP3} &2.90 & 2.57 & 425	 & 1020	& 5	& 19& 3182& 46.17\% \\
			\hline
	\end{tabular}\label{easycalls}
\end{table}

As it can be seen in Table \ref{easycalls}, using Propositions \ref{propcase1}-\ref{propcase4}, the running time of the separation algorithm decreases considerably. 
This is due to the fact that, among the total number of the calls of the exact separation algorithm, 66.42\%, 31.72\%, and 46.17\% of them can be computed directly using 
Propositions \ref{propcase1}-\ref{propcase4} for testsets NDP1, NDP2, and NDP3 without calling the time-consuming row generation subroutine.
Since the models in testset NDP1 contain only a single facility, the conditions in Propositions \ref{propcase1} and \ref{propcase2} occur frequently while those in Propositions \ref{propcase3} and \ref{propcase4} never occur.
Compared with those in Propositions \ref{propcase1} and \ref{propcase2}, the conditions in Propositions \ref{propcase3} and \ref{propcase4} occur less frequently.
This is not surprising since the conditions in Propositions  \ref{propcase3} and \ref{propcase4} are much stricter than those in Propositions \ref{propcase1} and \ref{propcase2}.
Finally, we observe that the reduction on the running time is not consistent with that of the successful rate. This is because, using the row generation subroutine, the separation problems that fulfill one of the conditions in Propositions \ref{propcase1}-\ref{propcase4} are easier to be solved than those that do not fulfill any of them.

\subsection{Performance effect of using Algorithm \ref{acceleration}}
Finally, we report the performance effect of employing Algorithm \ref{acceleration} in the exact separation procedure.
Compared with EXACT, NOALG2 refers to the setting of the exact separation procedure 
without using Algorithm \ref{acceleration}.
Notice that this may lead to a weaker constraint \eqref{cons2} in the row generation subroutine.
For each setting, we report the running time of the row generation subroutine and the average iteration of a row generation call in columns ``rgtime"  and ``iter", respectively. We present the average results among all the models.  
\begin{table}[ht]
	\centering
	\setlength{\tabcolsep}{3pt} 
	\renewcommand{\arraystretch}{1.2} 
	\centering
	\small 
	\caption{Performance effect of using Algorithm \ref{acceleration}.}
	\begin{tabular}{|c|cc|cc|}
		\hline
		\multirow{2}{1cm}{Testset}& \multicolumn{2}{c|}{NOALG2} & 
		\multicolumn{2}{c|}{EXACT} \\\cline{2-5}
		& rgtime & iter & rgtime & iter  \\ 
		\hline
		{NDP1} &1.42&8.87 &0.80& 4.63  \\
		{NDP2} &144.20& 61.62 & 91.10 & 21.83 \\
		{NDP3} &3.92& 9.22 & 2.57 & 6.36 \\
		\hline
	\end{tabular}\label{acc}
\end{table}

As it can be seen in Table \ref{acc}, with Algorithm \ref{acceleration}, the iteration of row generation subroutine per call decreases significantly, which in turn, reduces 43.66\%, 36.77\%, and 34.44\% of the row generation time in testsets NDP1, NDP2, and NDP3, respectively. This confirms that our proposed algorithm indeed works well in practice. Notice that it is reasonable to observe reductions on the average iteration exceeds reduction on the running time since Algorithm \ref{acceleration} also leads to a denser constraint for the linear programming problem \eqref{partialLP1}.

\section{Conclusion and future work} \label{sect:conclusionandfuturework}
In this paper, we have considered the separation problem of the flow arc-set polyhedron in the unsplittable capacitated network design problem.
By solving the separation problem, we generated the facet-defining inequality for the considered polyhedron.
We showed that in some special cases, a closed form of the separation problem can be derived.
For the general case, we used the exact separation algorithm to solve the separation problem.
Moreover, a new technique was proposed to reduce the computational time in the row generation subroutine of the exact separation algorithm.
The numerical experiments showed the effectiveness of the exact separation algorithm in solving the unsplittable capacitated network design problem and the advantage of the proposed technique in reducing the exact separation time.

There still exist some ideas to be explored in this study. For the flow arc-set polyhedron, we have proposed a new technique to speed up the row generation subroutine; see Algorithm \ref{acceleration}. It deserves to test whether or not the same improvement can be observed for the knapsack polyhedron. In this study, we only implemented the exact separation algorithm on the flow arc-set polyhedron, but it can be extended on the cut-set polyhedron, see for example \cite{Achterberg2010}. We are currently conducting this topic to see the performance effect of generalizing this to the cut-set polyhedron.

\bibliographystyle{spmpsci}      
\bibliography{exactseparation}   

\newpage
\section*{Appendix}

\begin{lemma}
	\label{lemma2}
	Suppose that $T=\{1\}$ and inequality \eqref{facet} with $\beta_1 = 1$ is a facet-defining {inequality} of polyhedron $ P $.
	For each $ q \in Q $, if $ a_q \leq b_1  $, then $ 0 \leq \alpha_q \leq 
	1 $.
\end{lemma}
\begin{proof}
	Combining with $a_q\leq b_1$, $\beta_1 = 1$, and Proposition \ref{facettheorem},
	we have the statement.\qed
\end{proof}

\begin{lemma} \label{lpclosedform1}
 	Given $ (\x, \y_1) \in {[0,1]}^{|Q|} \times \mathbb{R}_+ $, the linear programming problem
 	\begin{align}                                                                                           
 	&  \max_{\alpha,\gamma} \left\{\sum_{q \in Q} \x_q \alpha_q- \y_1 - \gamma: -r-\gamma\leq                   
 	0, \right. \nonumber \\                                                                                 
 	& \qquad\qquad\qquad   \left.\sum_{q \in Q}\alpha_q - (r+1)  -\gamma \leq  0,~0\leq  \alpha_q \leq 1, ~\forall~ q \in Q  \vphantom{\sum_{q \in Q} \x_q \alpha_q}\right\} \label{easycases}     
 	\end{align}
 	has an optimal solution $ (\boldsymbol{e^d}, -r) $ where $r$ and $ d $ are defined in \eqref{rdef} and \eqref{ddef}, respectively.
\end{lemma}
\begin{proof}
	By simple calculation, point $(\boldsymbol{e^d}, -r)$ is a 
	feasible solution of problem \eqref{easycases} with the objective 
	value being $ \x_d -\y_1 +r $. It is optimal since
	$$\begin{aligned}
	&~\sum_{q \in Q} \x_q \alpha_q-\y_1 - \gamma &&\\
	\leq &~  \x_d \sum_{q \in 
		Q}\alpha_q -\y_1 -\gamma  &&(\text{from}~\alpha_q \geq 0~{\text{and the definition of }~d~\text{in}~\eqref{ddef}})\\
	\leq&~  \x_d[\gamma + (r + 1)] -\y_1 -\gamma && (\text{from 
		$ \sum_{q \in Q}\alpha_q - (r+1)  -\gamma \leq  0 $ in problem \eqref{easycases}}) \\
	= &~\x_d + \x_dr  - (1-\x_d)\gamma - \y_1 &&\\
	\leq&~ \x_d + \x_dr  + (1-\x_d)r - \y_1 &&(\text{from $ -r  -\gamma\leq 0 $ 
		in problem \eqref{easycases}}) \\
	= &~ \x_d -\y_1 +r.&& \qquad\qquad\qquad\qquad\xqedhere{118.5pt}{\qed} 
	\end{aligned}  $$
\end{proof}

\section*{Proof of Proposition \ref{propcase1}}
\begin{proof}
	From the definition of $r$ in \eqref{rdef}, we have $ (\boldsymbol 0,r) \in 
	X $. This, combined with assumption (ii) and condition  
	\eqref{first}, implies that
	$$X = \big\{(\boldsymbol 0,r)\big \} \cup \big \{(x,k) \in \{0,1\}^{|Q|} \times \mathbb{Z}_+ \,:\, k \geq r+1\big \}.$$
	It is obvious  that the vertices of polyhedron $ P $, i.e., $\conv(X)$, are points
	$(\boldsymbol 0,r)$ and $(x,r+1)$ for all $x\in \{0,1\}^{|Q|}$ with 
	$x\neq \boldsymbol0$. Particularly, in problem \eqref{seppromain}, the vertex $ (\boldsymbol{e},r+1) $ 
	corresponds to the constraint
	\[\sum_{q \in Q}\alpha_q-(r+1)-\gamma \leq 0.\label{allecons}\]
	By removing the constraints that are dominated by \eqref{allecons} and using 
	Lemma \ref{lemma2} and Theorem \ref{solisfacet}, problem \eqref{seppromain} is further equivalent to problem \eqref{easycases}.
	Hence, by Lemma \ref{lpclosedform1}, in this case, point $(\boldsymbol{e^d}, -r)$ is optimal for problem \eqref{seppromain}, which corresponds to the  inequality $x_d 
	\leq y_1 -r$ of polyhedron $ P $.
	Furthermore, inequality $x_d \leq y_1 -r$ is facet-defining for polyhedron $ P $ since the $|Q|+1$ affinely independent points $(\boldsymbol{0},r)$, $(\boldsymbol{e^d},r+1)$,  and $(\boldsymbol{e^d} + \boldsymbol{e^q} ,r+1)$ for each $q\in Q\backslash \{d\}$ are on the face $\{(x,y_1)\in P: x_d = y_1-r\}$. \qed
\end{proof}

\begin{lemma} \label{lpclosedform2}
	Let $ r $ and $d$ be the values defined in \eqref{rdef} and \eqref{ddef}, respectively, and $ (\x, \y_1) \in {[0,1]}^{|Q|} \times \mathbb{R}_+ $. Consider the following linear programming problem
	\begin{align}
	&  \max_{\alpha,\gamma} \left\{\sum_{q \in Q} \x_q \alpha_q- \y_1 - \gamma: -r-\gamma\leq 
	0, \right. \nonumber \\
	& \qquad\left.\sum_{q\in Q\backslash\{\q\} }\alpha_q - (r+1)  -\gamma \leq  0,\ \forall ~\q \in 
	Q, ~0\leq  \alpha_q \leq 1, ~\forall~ q \in Q  \vphantom{\sum_{q \in Q} \x_q \alpha_q}\right\} \label{easycases2}.
	\end{align}	
	The following results hold.
	\begin{enumerate}
		\item 
		[ {\rm(a)}]  If $ |Q| \leq 2$,  one of the optimal solutions of problem \eqref{seppromain} is $(\boldsymbol{e}, -r)$;
		\item 
		[{\rm(b)}] If $ |Q| \geq 3 $, one of the optimal solutions of problem \eqref{seppromain} is
				$$\left\{\begin{aligned}
				&(\boldsymbol{e^{d}}, -r), && \text{if}\  \frac{\sum_{q\in Q}\x_q}{|Q|-1} \leq 
				\x_d;\\
				&(\frac{1}{|Q|-1}\boldsymbol{e}, -r), && \text{if}\ \x_d < 
				\frac{\sum_{q\in Q}\x_q}{|Q|-1}\leq 1;\\
				&(\boldsymbol{e},-r+|Q|-2),\ &&\text{if}\ \frac{\sum_{q\in 
				Q}\x_q}{|Q|-1}> 1.
				\end{aligned} \right.$$ 
	\end{enumerate}
\end{lemma}
\begin{proof}
	 If $|Q| \leq  2$, by removing the redundant constraints that are dominated by $ -r -\gamma \leq 0 $ and $ \alpha_q \leq 1$, $ q \in Q $, problem \eqref{easycases2} is equivalent to
	$$\d\max_{\alpha,\gamma} \left \{ \sum_{q \in Q} \x_q \alpha_q- \y_1 - \gamma: -r-\gamma\leq 
	0,~ 0\leq  \alpha_q \leq 1, ~\forall~  q \in Q \right \}.$$ 
	Clearly, $(\boldsymbol{e}, -r)$ is an optimal solution. This proves case (a) in the statement.
	
	Next, we consider the case $|Q|\geq 3$.
	Since the coefficient of $\gamma$ in the objective function in problem
	\eqref{easycases2} is $-1$, 
	optimality of problem \eqref{easycases2} requires that $\gamma$ must be equal to $-r$, or $ \sum_{q\in Q\backslash\{q'\}}\alpha_q - (r+1)$ for some $q' \in Q $. 
	We have the following two cases.
	\begin{itemize}
		\item[1)] 
		$\gamma = -r $. 
		By eliminating the variable $\gamma$, problem \eqref{easycases2} reduces to
		\begin{align*}
		\max_{\alpha}\left\{\sum_{q \in Q} \x_q\alpha_q- \y_1 +r: \sum_{q\in Q\backslash\{\q\}}\alpha_q \leq 1,~\forall~ \q \in Q, ~ 0\leq  \alpha_q \leq 1,~  \forall~  q \in Q \right\}.   
		\end{align*}
		From the linear programming theory, there are at least $|Q|$ 
		constraints in the above problem being active at the basic 
		optimal solution. 
		\begin{itemize}
			\item[1.1)] If $\sum_{q\in Q\backslash\{\q\}}\alpha_q = 1$ for all $ \q \in 
			Q $, we have $\alpha_q =\frac{1}{|Q|-1}$ for all $ q \in Q$, and 
			$(\frac{1}{|Q|-1}\boldsymbol e, -r)$ is an optimal solution of 
			problem \eqref{easycases2}.
			\item[1.2)] Otherwise, we have $ \alpha_{q'} = 0 $ or  $ \alpha_{q'} = 1 $ for some $ q' \in Q $. In any case, there must exist some $ q'' \in Q $ such that $ \alpha_{q''} = 0 $. Together with $ \sum_{q\in Q\backslash\{q''\}} \alpha_q \leq 1 $, it can be easily verified that $(\boldsymbol{e^d}, -r)$ is an optimal solution of problem \eqref{easycases2}.
		\end{itemize}
		\item[2)] 
		$\gamma =  \sum_{q\in Q\backslash\{q'\}}\alpha_q - (r+1) $ for some $ q'\in Q $. 
		By eliminating the variable $\gamma$, problem \eqref{easycases2}  
		reduces to
		\begin{align}
		& \max_{\alpha} \left\{-\sum_{q\in Q\backslash\{q'\}}(1-\x_q)\alpha_q+\x_{q'}\alpha_{q'} - 
		\y_1   + ( r+1):\sum_{q\in Q\backslash\{q'\}}\alpha_q \geq 1, \right.  \nonumber \\
		& \qquad \qquad\qquad\alpha_{q} \geq \alpha_{q'},\ \forall ~q\in Q \backslash\{q'\}, ~0\leq  \alpha_q \leq 1, ~\forall~ q \in Q  \left.\vphantom{\sum_{q \in Q} \x_q \alpha_q}\right\} 	\label{eq1}. 
		\end{align}
		\begin{itemize}
			\item[2.1)] If $\alpha_{q'} = 0$, the constraints $\alpha_{q} \geq 
			\alpha_{q'},\ q \in Q \backslash \{q'\}$, are redundant. 
			This implies that $(\boldsymbol{e^{d'}}, -r)$ is an optimal solution of problem \eqref{eq1} where $d' \in \argmax_{q\in Q \backslash \{q'\}}\{ \x_q \}$.  
			We note that by the definition of $d$ in \eqref{ddef}, the objective value of problem \eqref{easycases2} at point $(\boldsymbol{e^{d'}}, -r)$ cannot be better than that at point  $(\boldsymbol{e^d}, -r)$.
			\item[2.2)]  If $\alpha_{q'} = 1$, we have $\alpha_q=1,$ for all 
			$q\in Q\backslash\{q'\}$.
			Then $(\boldsymbol e, -r+|Q|-2)$ is an optimal solution of problem \eqref{eq1}.
			\item[2.3)] If $0 < \alpha_{q'} < 1$, we have $\alpha_q \geq \alpha_{q'} > 0$ for all $q \in Q \backslash \{q'\}$.
			 There exists a basic optimal solution such that at least $|Q|$ constraints in problem \eqref{eq1} are active.
			By a simple analysis, the active constraints must be $\sum_{q\in Q\backslash\{q'\}}\alpha_q = 1 $ and 
			$\alpha_{q}=\alpha_{q'}$ for all $q\in Q\backslash\{q'\}$. This implies $\alpha_q 
			=\frac{1}{|Q|-1}$ for all  $q \in Q$, and hence
			$(\frac{1}{|Q|-1}\boldsymbol e, -r)$ is an optimal solution of 
			problem \eqref{easycases2}.
		\end{itemize}
	In summary, for problem \eqref{easycases2}, there are three potentially optimal solutions $(\boldsymbol{e^d}, -r)$, 
	$(\frac{1}{|Q|-1}\boldsymbol e, -r)$, and $(\boldsymbol e, -r+|Q|-2)$ with the objective value $ \x_d - \y_1 +r $, $\frac{1}{|Q| -1} \sum_{q \in Q}\x_q - \y_1 + r$, and $  \sum_{q \in Q}\x_q - \y_1 + r  - |Q| + 2 $, respectively. Finally, comparing these three values, we have case (b) in the statement. This completes the proof. \qed
	\end{itemize}
\end{proof}

\section*{Proof of Proposition \ref{propcase2}}
\begin{proof}
	From the definition of $r$ in \eqref{rdef}, we have $ (\boldsymbol 0,r) \in 
	X $. This, together with assumption {\rm(ii)} and condition \eqref{third}, 
	implies that
	$$\begin{aligned}
	X = & \big\{(\boldsymbol 0,r)\big\}\cup \big\{(x, k) \in \{0,1\}^{|Q|} \times \mathbb{Z}_+\,: \,x\neq \boldsymbol e,~k=r+1\big \}\\
	& \qquad\qquad\qquad\qquad \cup \big \{(x,k)\in \{0,1\}^{|Q|} \times \mathbb{Z}_+\,:\, k\geq r +2 \big \}.
	\end{aligned}$$
	The vertices of {polyhedron} $ P $, i.e., $ \conv(X) $, are $(\boldsymbol 0,r)$, $(x,r+1)$  for all $x\in \{0,1\}^{|Q|}$ with $x\neq \boldsymbol 0$ and $x\neq \boldsymbol e$, and $(\boldsymbol e,r+2)$. 
	Similar to the proof in Proposition \ref{propcase1}, by removing redundant 
	constraints and using Lemma \ref{lemma2} and Theorem \ref{solisfacet}, problem \eqref{seppromain} reduces 
	to problem \eqref{easycases2}.
	Hence, if $|Q| \leq  2$, by Lemma \ref{lpclosedform2}, the optimal solution $(\boldsymbol{e}, -r)$ of problem \eqref{easycases2} corresponds to the inequality $\sum_{q \in Q} x_q \leq y_1-r$ of polyhedron $ P $.
	The associated face $\{(x,y)\in P: \sum_{q \in Q} x_q= y_1-r\}$ contains $ |Q| + 1 $ affinely independent points: $(\boldsymbol{0},r)$ and $(\boldsymbol{e^q},r+1)$ for each $ q \in Q $, which shows that inequality $\sum_{q \in Q} x_q \leq y_1-r$ defines a facet of polyhedron $ P $. This proves case (a) in the statement.
	
	Analogously, if $|Q| \geq  3$, by Lemma \ref{lpclosedform2}, points $(\boldsymbol{e^d}, -r)$, 
	$(\frac{1}{|Q|-1}\boldsymbol e, -r)$, and $(\boldsymbol e, -r +|Q|-2)$ are three potentially optimal solutions for problem \eqref{easycases2} which correspond to inequalities $x_{d}\leq  y_1-r$, $\frac{1}{|Q|-1}\sum_{q \in Q} x_q\leq y_1 -r$, and $\sum_{q\in Q}x_q\leq  y_1-r+|Q|-2$, respectively. To prove that each of the three inequalities defines a facet of polyhedron $P$, we list the $|Q|+1$ affinely independent points in polyhedron $ P $ fulfilling them at equality in the following.
\begin{table}[ht]
	\centering
	\setlength{\tabcolsep}{3pt} 
	\renewcommand{\arraystretch}{2} 
	\centering
	\small 
	\begin{tabular}{|l|l|}
		\hline
		$x_{d}\leq  y_1-r$ & $(\boldsymbol{0}, r)$, $(\boldsymbol{e^d}, r+1)$, $(\boldsymbol{e^d}+\boldsymbol{e^q}, r+1)$ for each  $q\in Q\backslash\{d\}$\\
		\hline
		$\frac{1}{|Q|-1}\sum_{q \in Q} x_q\leq y_1 -r$& $(\boldsymbol{0}, r)$, $(\boldsymbol{e}-\boldsymbol{e^q}, r+1)$ for each  $q\in Q$\\
		\hline
		$\sum_{q\in Q}x_q\leq  y_1-r+|Q|-2$ &  $(\boldsymbol{e}-\boldsymbol{e^q}, r+1)$ for each $q\in Q$, $(\boldsymbol{e}, r+2)$\\
		\hline
	\end{tabular}
\end{table}

\noindent Thus, we have case (b) in the statement. This completes the proof.
	\qed
\end{proof}

\section*{Proof of Proposition \ref{propcase3}}
\begin{proof}
	From the definition of $r$ in \eqref{rdef}, we have $ (\boldsymbol 0,r \boldsymbol{f^1}) \in 
	X $. Combining with assumptions {\rm(ii)}, {\rm(iii)}, and condition \eqref{first}, we can write set $X$ as:
	$$\begin{aligned}
	X = \bigg \{(\boldsymbol 0,r \boldsymbol{f^1})\bigg\}\bigcup \bigg\{(x, 
	k) \in \{0,1\}^{|Q|} \times \mathbb{Z}_+^{|T|}\,:\, k_1\geq r +1,~ \sum_{t \in T\backslash \{1\}}k_t = 0  \bigg  \} &\\
	\bigcup \bigg \{(x, 
	k) \in \{0,1\}^{|Q|} \times \mathbb{Z}_+^{|T|}\, :\, \sum_{t \in T \backslash\{1\}} k_t \geq 1\bigg \}.&
	\end{aligned}$$
	Clearly, if $r=0$, the vertices of polyhedron $P$, i.e., $\conv(X)$, are  $(\boldsymbol 0,r\boldsymbol{f^1})$, $(x,(r+1)\boldsymbol{f^1})$ 
	for all $x\in \{0,1\}^{|Q|}$ with $x\neq \boldsymbol 0$, and $(x,\boldsymbol{f^t})$ for all $x\in \{0,1\}^{|Q|}$ with $x \neq \boldsymbol{0}$ and $t\in T\backslash\{1\}$.
	If $r>0$, the additional vertices of polyhedron $P$ are $(\boldsymbol{0},\boldsymbol{f^t})$ for all $t\in T\backslash\{1\}$.
	Similar to the proof in Proposition \ref{propcase1}, by removing redundant 
	constraints and using Lemma \ref{lemma2} and Theorem \ref{solisfacet}, problem \eqref{seppromain} reduces to
	\begin{align}
	 & \max_{\alpha,\beta,\gamma} \left\{\sum_{q \in Q} \x_q \alpha_q -\y_1 -\sum_{t\in T\backslash\{1\}}\y_t\beta_t - \gamma:
	-r -\gamma\leq 0, \right.\qquad \qquad\qquad\; \nonumber \\
	& \qquad \sum_{q \in Q }\alpha_q - (r+1)  -\gamma \leq  0,~
	\sum_{q \in Q}\alpha_q-\beta_t - \gamma\leq 0,~ \forall\ t\in T\backslash\{1\}, \nonumber \\
	 & \qquad\qquad\qquad\beta_1 = 1, ~ \beta_t\geq 0,~\forall\ t\in T\backslash\{1\}, ~0\leq  \alpha_q \leq 1, ~\forall~ q \in Q   \left.\vphantom{\sum_{q \in Q} \x_q \alpha_q}\right\}. \label{easycases3}  
	\end{align}
	We now relax the bound constraints $\beta_t \geq 0$ for all $t\in T\backslash\{1\}$ in problem \eqref{easycases3}.
	Then as the objective coefficient of $\beta_t$ in problem 
	\eqref{easycases3} 
	is $-\y_t \leq 0$, we have 
	$\beta_t =  \sum_{q \in Q}\alpha_q -\gamma$ for all $t\in T\backslash\{1\}$ in the relaxation problem. 
	Substituting them into the objective function and dividing the objective function by the positive value $ 1-\sum_{t\in T\backslash\{1\}}\y_t $, we obtain an equivalent relaxation problem:
	\begin{align}
	  \max_{\alpha,\gamma} \left\{\sum_{q \in Q} \frac{\x_q- 
	  	\sum_{t\in T\backslash\{1\}}\y_t}{1-\sum_{t\in T\backslash\{1\}}\y_t} \alpha_q -
	\gamma - \frac{\y_1}{1-\sum_{t\in T\backslash\{1\}}\y_t}:
	 \right.\qquad~~~~& \nonumber \\
	 -r -\gamma\leq 0, ~\sum_{q \in Q }\alpha_q - (r+1)  -\gamma \leq  0, 
	 ~0\leq  \alpha_q \leq 1, ~\forall~ q \in Q  &\left.\vphantom{\sum_{q \in Q} \x_q \alpha_q}\right\}. \label{easycases3sub}  
	\end{align}
	If, for some $ q\in Q$, variable $ \alpha_q $'s objective coefficient  $\frac{\x_q- 
		\sum_{t\in T\backslash\{1\}}\y_t}{1-\sum_{t\in T\backslash\{1\}}\y_t} \leq 0$, then there must exist an optimal solution of problem \eqref{easycases3sub} such that $\alpha_q = 0$.
	Hence, we can remove the variables $ \alpha_q $ with nonpositive objective coefficients ($ q\in Q\backslash\tilde{Q} $ where $\tilde{Q}$ is defined in \eqref{Qtilde}) from {problem} \eqref{easycases3sub} and concentrate on the equivalent form of problem \eqref{easycases3sub}: 
	\begin{align}
	\max_{\alpha,\gamma} \left\{\sum_{q \in \tilde{Q}} \frac{\x_q- 
		\sum_{t\in T\backslash\{1\}}\y_t}{1-\sum_{t\in T\backslash\{1\}}\y_t} \alpha_q -
	\gamma - \frac{\y_1}{1-\sum_{t\in T\backslash\{1\}}\y_t}:
	\right.\qquad~~~~& \nonumber \\
	-r -\gamma\leq 0, ~\sum_{q \in \tilde{Q} }\alpha_q - (r+1)  -\gamma \leq  0, 
	~0\leq  \alpha_q \leq 1, ~\forall~ q \in \tilde{Q}  &\left.\vphantom{\sum_{q \in Q} \x_q \alpha_q}\right\}. \label{easycases3subs}  
	\end{align}
	We have two following cases. 
	\begin{itemize}
		\item[1)] $\tilde{Q} = \varnothing$. It can be easily verified that point $ (\alpha, \gamma)= (\boldsymbol{0}, -r) $ is optimal for problem \eqref{easycases3sub}. 
		Together with $ \beta_t =\sum_{q \in Q}\alpha_q -\gamma = r \geq 0 $ for all $ t \in T \backslash\{1\} $, we know that $(\boldsymbol{0}, \boldsymbol{f^1} + r\sum_{t \in T \backslash\{1\}} \boldsymbol{f^t}, -r)$ is an optimal solution for problem \eqref{easycases3} corresponding to the inequality $  0 \leq y_1 + r\sum_{t\in T \backslash\{1\}}y_t-r  $ of polyhedron $ P $.
		Moreover, the inequality defines a facet of polyhedron $ P $ since the $|Q|+|T|$ affinely independent points $(\boldsymbol{0},r\boldsymbol{f^1})$, $(\boldsymbol{0},\boldsymbol{f^t})$ for each $t\in T\backslash\{1\}$, and $(\boldsymbol{e^q},\boldsymbol{f^{t'}})$ for each $q\in Q$ and some $ t' \in T \backslash\{1\} $, are on the face $\{(x,y)\in P: 0 = y_1 + r\sum_{t\in T \backslash\{1\}}y_t-r\}$.
		\item[2)] $\tilde{Q} \neq \varnothing$.
		Notice that problem \eqref{easycases3sub} are a form of problem \eqref{easycases} and hence by Lemma \ref{lpclosedform1}, point $ (\alpha, \gamma) = (\boldsymbol{e^d}, -r) $ is optimal for problem \eqref{easycases3sub} where $ d $ is defined in \eqref{ddef}. 
		Furthermore, for all $ t \in T$, we have $ \beta_t  =\sum_{q \in Q}\alpha_q -\gamma =  r+1 \geq 0  $ showing that $(\boldsymbol{e^{d}}, \boldsymbol{f^1} + (r+1)\sum_{t \in T \backslash\{1\}} \boldsymbol{f^t}, -r)$ is an optimal solution of problem \eqref{easycases3} which corresponds to the inequality $  x_d \leq y_1 + (r+1)\sum_{t\in T \backslash\{1\}}y_t-r  $ for $ P $.
		Finally, the $|Q|+|T|$ affinely independent points 
		$(\boldsymbol{0},r\boldsymbol{f^1})$, $(\boldsymbol{e^d},(r+1)\boldsymbol{f^1})$,
		$(\boldsymbol{e^d} + \boldsymbol{e^q},(r+1)\boldsymbol{f^1})$ for each $q\in Q \backslash \{d\}$,
		and $(\boldsymbol{e^{d}},\boldsymbol{f^t})$ for each $t\in T\backslash\{1\}$ are on the face $\{(x,y)\in P: x_d = y_1 + (r+1)\sum_{t\in T \backslash\{1\}}y_t-r\}$, which implies that the inequality $x_d \leq y_1 + (r+1)\sum_{t\in T \backslash\{1\}}y_t-r$ is facet-defining for polyhedron $P$.
		\qed
	\end{itemize} 
\end{proof}

\section*{Proof of Proposition \ref{propcase4}}
\begin{proof}
	From the definition of $r$ in \eqref{rdef}, we have $ (\boldsymbol 0,r\boldsymbol{f^1}) \in 
	X $. It follows from assumptions {\rm(ii)}, {\rm(iii)} and condition \eqref{third} that 
	the set $X$ can be equivalently written as:
	$$\normalsize\begin{aligned}
	X = \bigg \{(\boldsymbol 0,r \boldsymbol{f^1})\bigg\}\bigcup \bigg\{(x, 
	k) \in \{0,1\}^{|Q|}\! \times\! \mathbb{Z}_+^{|T|}\!:\! x \neq \boldsymbol{e},~k_1= r +1,\!\!\sum_{t \in T\backslash \{1\}}\!\!k_t = 0&\bigg  \} \\
	\bigcup \bigg \{(x, 
	k) \in \{0,1\}^{|Q|} \times \mathbb{Z}_+^{|T|} :k_1\geq r +2,~\!\!\sum_{t \in T\backslash \{1\}}\!\!k_t = 0&\bigg \}\\
	 \bigcup \bigg \{(x, 
	k) \in \{0,1\}^{|Q|} \times \mathbb{Z}_+^{|T|} :\sum_{t \in T\backslash \{1\}}\!\!k_t \geq 1&\bigg \}.
	\end{aligned}$$
	Clear, if $r=0$, the vertices of {polyhedron} $ P $, i.e., $ \conv(X) $, are $(\boldsymbol 0,r\boldsymbol{f^1})$, $(x,(r+1)\boldsymbol{f^1})$ 
	for all $x\in \{0,1\}^{|Q|}$ with $x\neq \boldsymbol 0$ and $x\neq \boldsymbol e$, $(\boldsymbol e, (r+2)\boldsymbol{f^1})$, and $(x,\boldsymbol{f^t})$ for all $x\in \{0,1\}^{|Q|}$ with $x \neq \boldsymbol{0}$ and for all $t\in T\backslash\{1\}$.
	If $r>0$, the additional vertices of polyhedron $P$ are $(\boldsymbol{0},\boldsymbol{f^t})$ for all $t\in T\backslash\{1\}$.
	Similar to the proof in Proposition \ref{propcase1}, by removing redundant 
	constraints and using Lemma \ref{lemma2} and Theorem \ref{solisfacet}, problem \eqref{seppromain} reduces to
	\begin{align}
	&\max_{\alpha,\beta,\gamma} \left\{\sum_{q \in Q} \x_q \alpha_q -\y_1 -\sum_{t\in T\backslash\{1\}}\y_t\beta_t - \gamma: -r -\gamma\leq 0, \right. \nonumber \\
	&\sum_{q \in Q\backslash\{\q\} }\!\!\!\!\alpha_q - (r+1)  -\gamma \leq  0,~ \forall~\q\in Q,~ \sum_{q \in Q}\!\alpha_q-\beta_t - \gamma\leq 0,~ \forall\ t\in T\backslash\{1\}, \nonumber \nonumber \\
	&~~~~\qquad\qquad\qquad\beta_1 = 1,~\beta_t\geq 0,~\forall\ t\in T\backslash\{1\}, ~0\leq  \alpha_q \leq 1, ~\forall~ q \in Q \left.\vphantom{\sum_{q \in Q} \x_q \alpha_q}\right\}.\label{easycases40}
	\end{align}
	We now consider the relaxation of problem \eqref{easycases40} obtained by relaxing the bound constraints $\beta_t \geq 0$ for all $t\in T\backslash\{1\}$.
	As the objective coefficient of $\beta_t$ in problem 
	\eqref{easycases40} 
	is $-\y_t \leq 0$, we can set 
	$\beta_t =  \sum_{q \in Q}\alpha_q -\gamma$ for all $t\in T\backslash\{1\}$ in the relaxation problem. 
	In analogy to the proof in Proposition \ref{propcase3}, substituting them into the objective function and dividing the objective function by the positive value $ 1-\sum_{t\in T\backslash\{1\}}\y_t $, this relaxation problem reduces to: 
	\begin{align}
	&\max_{\alpha,\gamma} \left\{ \sum_{q \in {Q}} 
	\frac{\x_q- \sum_{t\in T\backslash\{1\}}\y_t}{1-\sum_{t\in T\backslash\{1\}}\y_t} \alpha_q - 
	\gamma - \frac{\y_1}{1-\sum_{t\in T\backslash\{1\}}\y_t}: -r -\gamma\leq 0, 
	\right.\nonumber \\
	&~~\qquad\sum_{q \in {Q}\backslash\{\q\} }\alpha_q - (r+1)  -\gamma \leq  0,~ \forall~\q\in {Q},~
	0\leq  \alpha_q \leq 1, ~\forall~ q \in {Q} \left.\vphantom{\sum_{q \in {Q}} \x_q \alpha_q}\right\}.\label{easycases4}
	\end{align}
	\begin{itemize}
		\item[1)] $\tilde{Q}\neq Q$. For $ q \in Q \backslash \tilde{Q} $ where $\tilde{Q}$ is defined in \eqref{Qtilde},  since $\frac{\x_q- \sum_{t\in T\backslash\{1\}}\y_t}{1-\sum_{t\in T\backslash\{1\}}\y_t} < 0 $, we can set $ \alpha_q = 0 $ in problem \eqref{easycases4}.
		This, together with $ \alpha_q \geq 0$ for all $ q \in \tilde{Q} $, implies that the $ |Q| $ constraints $ \sum_{q \in {Q}\backslash\{\q\} }\alpha_q - (r+1)  -\gamma \leq  0 $, for all $ \q \in Q $, in problem \eqref{easycases4}, can be reduced to a single constraint $ \sum_{q \in \tilde{Q}}\alpha_q - (r+1)  -\gamma \leq  0  $ since they are either equivalent to or dominated by this constraint.
		Therefore, in this case, problem \eqref{easycases4} reduces to problem \eqref{easycases3subs}, and repeating 1) and 2) in the proof of Proposition \ref{propcase3}, we have cases (a) and (b) in the statement.
		\item[2)] $\tilde{Q} = Q$ and 
		$|Q| \leq 2$. By Lemma \ref{lpclosedform2},  point $(\boldsymbol{e}, -r) $ is optimal for problem \eqref{easycases4}. For each $ t \in T\backslash\{1\}$, we have $ \beta_t  =\sum_{q \in Q}\alpha_q -\gamma = r+|Q| \geq 0 $ showing that $(\boldsymbol{e}, \boldsymbol{f^1} + (r+|Q|)\sum_{t \in T \backslash\{1\}} \boldsymbol{f^t}, -r)$ is an optimal solution of problem \eqref{easycases40} corresponding to the inequality $\sum_{q\in Q}x_q \leq y_1 + (r+|Q|)\sum_{t\in T \backslash\{1\}}y_t-r  $ for polyhedron $ P $. Moreover, to show that it is facet-defining for polyhedron $P$, we list the $|Q|+|T|$ affinely independent points in polyhedron $ P $ satisfying it on equality: $(\boldsymbol{0},r \boldsymbol{f^1})$, $(\boldsymbol{e^q},(r+1) \boldsymbol{f^1})$ for each $ q \in Q $, and $(\boldsymbol{e}, \boldsymbol{{f}^t})$ for each $ t \in T \backslash\{1\} $. 
			Thus, we have case (c) in the statement.

		\item[3)] $\tilde{Q} = Q$ and $|Q| \geq 3$. By Lemma \ref{lpclosedform2}, one of the three points $(\boldsymbol{e^d}, -r)$, $(\frac{1}{|Q|-1}\boldsymbol{e}, -r)$, and $(\boldsymbol{e}, -r+|Q|-2)$ are optimal for problem \eqref{easycases4}. Using these three points to compute $ \beta_t  =\sum_{q \in Q}\alpha_q -\gamma$, for all $t\in T\backslash\{1\}$, we have $\beta_t= r+1$, $r+\frac{|Q|}{|Q|-1}$, and $ r+2 $, respectively. In all three cases, we have $ \beta_t \geq 0 $ for all $ t \in  T \backslash \{1\} $ and hence one of the three points $(\boldsymbol{e^d},\boldsymbol{f^1}+(r+1)\sum_{t\in T\backslash\{1\}}\boldsymbol{f^t}, -r)$, $(\frac{1}{|Q|-1}\boldsymbol{e}, 
		\boldsymbol{f^1}+(r+\frac{|Q|}{|Q|-1})\sum_{t\in T\backslash\{1\}}\boldsymbol{f^t},-r)$, and $(\boldsymbol{e},\boldsymbol{f^1}+(r+2)\sum_{t\in T\backslash\{1\}}\boldsymbol{f^t}$, $-r+|Q|-2)$ must be optimal for problem \eqref{easycases40}. 
		Finally, the following table shows that the associated inequalities define facets of polyhedron $ P $.
		\begin{table}[H]
			\centering
			\setlength{\tabcolsep}{3pt} 
			\renewcommand{\arraystretch}{2} 
			\centering
			\small 
			\begin{tabular}{|l|l|}
				\hline
				\multirow{3}{4.5cm}{$x_{d}\leq  y_1 + (r+1)\sum_{y\in T\backslash\{1\}} y_t-r$} & $(\boldsymbol{0}, r\boldsymbol{f^1})$, $(\boldsymbol{e^d}, (r+1)\boldsymbol{f^1})$,\\ &$(\boldsymbol{e^d}+\boldsymbol{e^q}, (r+1)\boldsymbol{f^1})$ for each  $q\in Q\backslash\{d\}$,\\
				&$(\boldsymbol{e^d}, \boldsymbol{f^t})$ for each  $t\in T\backslash\{1\}$.\\
				\hline
				\multirow{2}{4.5cm}{$\frac{1}{|Q|-1}\sum_{q \in Q} x_q\leq y_1 +(r+\frac{|Q|}{|Q|-1})\sum_{y\in T\backslash\{1\}} y_t-r$}& $(\boldsymbol{0}, r\boldsymbol{f^1})$, $(\boldsymbol{e}-\boldsymbol{e^q}, (r+1)\boldsymbol{f^1})$ for each  $q\in Q$,\\
				&$(\boldsymbol{e}, \boldsymbol{f^t})$ for each  $t\in T\backslash\{1\}$.\\
				\hline
				$\sum_{q\in Q}x_q\leq y_1+$ & $(\boldsymbol{e}-\boldsymbol{e^q}, (r+1)\boldsymbol{f^1})$ for each  $q\in Q$, \\
				$(r+2) \sum_{y\in T\backslash\{1\}} y_t-r+|Q|-2$ & $(\boldsymbol{e}, \boldsymbol{f^t})$ for each  $t\in T\backslash\{1\}$, $(\boldsymbol{e}, (r+2)\boldsymbol{f^1})$.\\
				\hline
			\end{tabular}
		\end{table}
	 	Thus, we have case (d) in the statement. This completes the proof.\qed
	\end{itemize}
\end{proof}

\end{document}